\documentclass[12pt]{amsart}
\usepackage{amsthm,amssymb,amsmath,amscd,epic,eepic}
\usepackage[all]{xy}
\usepackage{xcolor}

\usepackage{enumitem}
\usepackage[normalem]{ulem}
\usepackage{cancel}
\usepackage{stmaryrd}

\usepackage{relsize}
\newcommand     {\comment}[1]   {}
\newcommand{\mute}[2] {}
\newcommand     {\printname}[1] {}

\newcommand{\todo}[2][1]{\vspace{1 mm}\par \noindent
\marginpar{\textsc{}} \framebox{\begin{minipage}[c]{#1
\textwidth} \tt #2 \end{minipage}}\vspace{1 mm}\par}
\newcommand{\sm}{\left(M,\omega\right)}

\newcommand{\rorbit}{\mathcal{O}} 
\newcommand{\torbit}{\mathfrak{O}} 

\numberwithin{equation}{section}
\newtheorem {Theorem}[equation]                   {Theorem}
\newtheorem*{Theorem*}                   {Theorem}

\newtheorem {Lemma}[equation]           {Lemma}

\newtheorem* {Corollary*}                {Corollary}
\newtheorem {Proposition} [equation]    {Proposition}

\newtheorem*{Question*}{Question}

\newtheorem* {Lemma*}                    {Lemma}

\theoremstyle{definition}
\newtheorem{Definition}[equation]{Definition}
\newtheorem*{Definition*}{Definition}

\theoremstyle{remark}
\newtheorem{Remark}[equation]{Remark}

\newtheorem*{Remark*}{Remark}

\newtheorem{Example}[equation]{Example}

\long\def\symbolfootnote[#1]#2{\begingroup%
\def\thefootnote{\fnsymbol{footnote}}\footnote[#1]{#2}\endgroup}

\def \Re {{\mathfrak R}}
\def \Im {{\mathfrak I}}

\def    \inv    {^{-1}}

\def \MmodT {{ M/\!/T}}
\def    \YmodT {{ Y/\!/ T}}

\def    \R	{{\mathbb R}}

\def    \C	{{\mathbb C}}
\def    \Z       {{\mathbb Z}}

\def    \h	{{\mathfrak h}}
\def    \ft	{{\mathfrak t}}
\def    \fh	{{\mathfrak h}}

\def	\ker 	{{\operatorname{ker}}}
\def	\exp	{{\operatorname{exp}}}

\begin{document}

\title[Connectedness beyond semitoric systems I]{Connectedness of fibers beyond semitoric systems I: the non-degenerate case}

\author[Daniele Sepe]{Daniele Sepe}
\address{Escuela de Matem\'aticas, Universidad Nacional de Colombia sede Medell\'in, Colombia.}
\email{dsepe@unal.edu.co}

\author[Susan Tolman]{Susan Tolman}
\address{Department of Mathematics, University of Illinois at Urbana-Champaign,
 Urbana, IL 61801.}
\email{tolman@illinois.edu}


\begin{abstract} 
In this paper we study the connectedness of the fibers of integrable systems that extend complexity one $T$-spaces with proper moment maps, assuming that every tall singular point is non-degenerate. Our main result states that if there are no tall singular points with a hyperbolic block and connected $T$-stabilizer, then each fiber is connected. Moreover, we prove that the above condition is necessary if either some reduced space is simply connected or the moment map for the integrable system is generic in a natural sense.
\end{abstract}

\maketitle

\tableofcontents
\section{Introduction}

A crucial step in studying Hamiltonian group actions is understanding the fibers of the moment map. In this direction, connectedness of the fibers has been established in the following cases:
\begin{enumerate}[label=(\arabic*),ref=(\arabic*)]
\item \label{item:connab} Hamiltonian group actions of compact abelian Lie groups with proper moment maps  \cite{At,GS2,LMTW}; and
\item \label{item:semitoric} semi-toric systems in dimension four and their natural generalization in dimension six  \cite{VN,Wa}.
\end{enumerate}

In both cases the connectedness of the fibers of the moment map is the first step towards classification results: For instance, it is central to proving convexity of the moment map image of Hamiltonian torus actions (see \cite{At,GS2,LMTW}), and to obtaining a convexity result for semi-toric systems in dimension four (see \cite{VN,Wa}).

In this paper we study the connectedness of the fibers of integrable systems that are extensions of complexity one torus actions,  assuming that every tall singular point is non-degenerate. We hope that this will eventually allow us to extend the classification of semi-toric systems in dimension four to higher dimensions. Complexity one torus actions are the simplest after the toric case and have been partially classified in \cite{KT1,KT2,KT3}. Non-degenerate singular points of integrable systems are generic, and can be thought of as the symplectic analog of critical points of Morse-Bott functions; they have attracted much interest, both for classical and quantum integrable systems (see \cite{BF,SV} and references therein). In particular, the integrable systems that we study include the semi-toric systems of \cite{VN}, and their higher dimensional generalization in \cite{Wa}.

In our main result, Theorem~\ref{thm:equivalence}, we give sufficient conditions for fibers of the moment map to be connected,
generalizing the results in \cite{VN,Wa}.  
Morse theory plays a central role in the theorems \ref{item:connab} and \ref{item:semitoric} above,  and our theorem both uses and  generalizes  the following simple Morse-theoretic result: Let $N$ be a closed, connected 2-dimensional manifold and let $\overline{g} \colon N \to \R$ be a Morse function. If $\overline{g}$ has no singular point of index one, then each fiber is connected (see Lemma \ref{Morse2}). We also explore the extent to which the hypotheses in our main result are necessary (see Theorem \ref{thm:converse}). \\

Before stating our results precisely, we begin with few definitions from integrable systems.

\begin{Definition}
  An {\bf integrable system} is a quadruple $\left(M,\omega, V,
    f\right)$, where $\sm$ is a connected
  $2n$-dimensional symplectic manifold, $V$ is an $n$-dimensional real
  vector space, and $f \colon M \to V^*$ is a smooth map  
   such that:
  \begin{enumerate}
  \item  the Poisson bracket $\big\{ \langle f, \xi_1 \rangle, \langle f, \xi_2 \rangle \big\} = 0$ for all $\xi_1, \xi_2 \in V$, and
  \item the function $f$ is regular on a dense subset of $M$.
  \end{enumerate}
  Here $\langle \cdot, \cdot \rangle \colon V^* \times V \to \R$ is the natural pairing.
\end{Definition}

The standard definition of integrable systems takes $V = \R^n$. In this paper, we assume that every integrable system is {\bf complete}, that is,  $f$ is
a moment map for a Hamiltonian $V$-action. \\

Fix an integrable system $\left(M,\omega, V,
    f\right)$. A point $p \in M$ is {\bf
  regular} if $D_pf$ has full rank and {\bf singular} otherwise. Given a point 
$p \in M$, the $V$-orbit $\rorbit$ through $p$ is isotropic. Hence, the vector space $\left(T_p
  \rorbit\right)^{\omega}/T_p\rorbit$ inherits a natural symplectic structure. 
Moreover, the stabilizer $W \subseteq V$ of $p$ acts by linear symplectomorphisms on $T_p M$ and fixes $T_p \rorbit.$
We call the induced symplectic action of $W$ on $(T_p\rorbit)^\omega/T_p \rorbit$
the {\bf linearized action} at $p$.

\begin{Definition}\label{defn:non-degenerate}
Let $p$ be a point in an integrable system $\left(M,\omega,V,f\right)$.
Let $W_0 \subseteq V$ be the Lie algebra of the stabilizer of $p$. The point $p$ is {\bf non-degenerate} if there exist identifications $ W_0 \simeq \R^{\dim W}$ and $(T_p\rorbit)^\omega/T_p \rorbit \simeq \C^{\dim W}$ such that the homogeneous moment map for the linearized action at $p$ is a product of the following ``blocks":
\begin{itemize}
    \item ({\bf elliptic}) $z \in \C \mapsto \frac{1}{2}|z|^2$;
    \item ({\bf hyperbolic}) $z \in \C \mapsto \Im(z^2)$;
    \item ({\bf focus-focus}) $(z_1,z_2) \in \C^2 \mapsto \left(\frac{1}{2}(|z_1|^2 - |z_2|^2),\Im(z_1z_2)\right)$.
\end{itemize}
Here $\Im(w)$ denotes the imaginary part of $w \in \C$ and $\C^{\dim W}$ has the standard symplectic structure. Moreover, $p$ has {\bf purely elliptic type} if $p$ is non-degenerate and all the blocks are elliptic.
\end{Definition}

In particular, by definition regular points are non-degenerate and have purely elliptic type.

\begin{Remark}\label{rmk:different_convention}
  The above definition of non-degeneracy is equivalent to the standard one (see, for instance, \cite[Theorem 4.8]{SV}). In particular,
  the focus-focus block in the above definition can be transformed to the  one in the
  integrable systems literature by applying a linear symplectomorphism.
\end{Remark}

\begin{Remark}\label{rmk:non-deg_Morse}
The notion of non-degeneracy can be viewed as a ``symplectic Morse-Bott" condition. In particular, in dimension two, a singular point is non-degenerate in the sense of Definition \ref{defn:non-degenerate} if and only if it is  non-degenerate  in the sense of Morse theory. Moreover, it is hyperbolic if and only if it has index one and elliptic otherwise. 
\end{Remark}

Next we introduce complexity one torus actions. Let $T$ be an $(n-1)$-dimensional torus with Lie algebra $\ft$. 

\begin{Definition}\label{compl_one}
A {\bf complexity one $\boldsymbol{T}$-space} is
a triple $(M,\omega,\Phi)$, where $(M,\omega)$ is a connected $2n$-dimensional
symplectic manifold
and $\Phi \colon M \to \ft^*$ is a moment map for 
an effective $T$-action.
\end{Definition}

Let $(M,\omega,\Phi)$ be a complexity one $T$-space. Given $\beta \in \ft^*$, the {\bf reduced space} at $\beta$ is the quotient $\Phi^{-1}(\beta)/T$, endowed with the subquotient topology; if $\beta = 0$, we write $\MmodT$ for the reduced space. 
A point $p \in M$
is {\bf short} if $[p]$ is an isolated point in the reduced space $\Phi^{-1}(\Phi(p))/T$ and is {\bf tall} otherwise. 
By \cite[Proposition 6.1]{KT1}, 
if $\Phi$ is proper each reduced space $\Phi^{-1}(\beta)/T$ 
containing more than one point is homeomorphic to (but not necessarily diffeomorphic to) a closed, connected oriented 2-dimensional manifold; moreover,  any two non-empty reduced spaces have the same fundamental group \cite[Lemma 5.7 and Corollary 9.7]{KT1}. 

 The main objects that we study in this paper are integrable systems that extend complexity one torus actions and have the property that every tall singular point is non-degenerate. Our main result is as follows.

\begin{Theorem}\label{thm:equivalence}
Let $\left(M,\omega, \ft \times \R, f:=(\Phi,g)\right)$ be an integrable system such that $(M,\omega,\Phi)$ is a complexity one $T$-space with a proper moment map. Assume that each tall 
singular point in $\left(M,\omega, \ft \times \R, f\right)$ is non-degenerate and that no such point has a hyperbolic block and connected $T$-stabilizer. Then
 the fiber $f^{-1}(\beta,c)$ is connected for each $(\beta, c) \in \ft^* \times \R$;
moreover,  the reduced space $\Phi^{-1}(\beta)/T$ is simply connected.
\end{Theorem}

In addition to semi-toric systems,
Theorem~\ref{thm:equivalence} 
applies to the following well-known integrable system.

\begin{Example}\label{exm:short-degenerate}
The  $1 \! \! : \! \! 2$ resonance is the integrable system $(\C^2,\omega, \R \times \R, f:=(\Phi,g))$, where $\omega \in \Omega^2(\C^2)$ is the standard symplectic form,
\[\Phi(z_1,z_2) := \textstyle \frac{1}{2}|z_1|^2 + |z_2|^2, \quad \text{and} \quad g(z_1,z_2) := \Re(z^2_1\overline{z}_2).\]
As explained in \cite[Sections 3.3.2 and 3.3.3]{CVN}, the origin, which is the only short point, is singular and degenerate.  Every other singular point in this integrable system is non-degenerate; moreover,  $z \in \C^2$ is a singular point with a hyperbolic block exactly if $z \in \{0\} \times \C^\times$,
that is, exactly if its $S^1$-stabilizer is $\Z_2 = \{\pm 1\}$.  
Therefore, the $1 \! \! : \! \! 2$ resonance satisfies the hypothesis of Theorem~\ref{thm:equivalence}. In particular, the fibers of $f$ are connected and each reduced space of the $S^1$-action is simply connected (see \cite[Remark 3.1]{CVN}).
\end{Example}

\begin{Remark}\label{rmk:wacheux}
    Theorem \ref{thm:equivalence} implies connectedness of the fibers for the higher dimensional analogs of semi-toric systems that Wacheux studied in \cite{Wa}: Integrable systems $\left(M,\omega, \ft \times \R, f:=\left(\Phi,g\right)\right)$ such that $\left(M,\omega,\Phi\right)$ is a complexity
  one $T$-space with a proper moment map all of whose singular points are non-degenerate without hyperbolic blocks. In particular, our main result generalizes \cite[Theorem 4.2.4]{Wa} to arbitrary dimensions.
\end{Remark}

\begin{Remark}
    Theorem \ref{thm:equivalence} is comparable to \cite[Theorems 1 -- 3]{PRV} in that all these results   assume that the singular points of an integrable system are non-degenerate,  disallow (certain) singular  points with a hyperbolic block, and prove that the fibers are connected.  However, the results in \cite{PRV}  do not require the first component of the moment map to generate a Hamiltonian circle action, but only hold in dimension four and have additional hypotheses on the moment map. 
\end{Remark}

Unfortunately, Theorem \ref{thm:equivalence} is not quite sharp. If a reduced space has positive genus, then $M$ must have non-degenerate singular points with a hyperbolic block and connected $T$-stabilizer; see Remark~\ref{highergenus}.
However, as we show below, it is still possible for 
all the fibers of $f$ to be connected in this case. 

\begin{Example} \label{exm:too_bad} Let $\omega$ be the standard symplectic form on the torus $N= \R^2/\Z^2$.
The function $g \colon N \to \R$ defined by $g([x,y]) = \cos(2 \pi x) + \cos (2 \pi y)$  
 is Morse and has connected fibers, a unique minimum and maximum, and two points of index one.
 Hence, $(N,\omega,\R, g)$ is an integrable system with  two hyperbolic singular points.
 (Since the compact torus $T$ is  trivial, all points  have connected $T$-stabilizer.)
To form higher dimensional examples, take the product of this integrable system with a compact, connected symplectic toric manifold.
\end{Example}
 
Nevertheless,  as our next result shows, Theorem~\ref{thm:equivalence} is quite close to being sharp.

\begin{Theorem}\label{thm:converse}
Let $\left(M,\omega, \ft \times \R, f:=(\Phi,g)\right)$ be an integrable system such that $(M,\omega,\Phi)$ is a complexity one $T$-space with a proper moment map. Assume that each tall 
singular point in $\left(M,\omega, \ft \times \R, f\right)$ is non-degenerate.  Moreover, assume that one of the following holds:
\begin{itemize} 
\item The reduced space $\Phi^{-1}(\beta_0)/T$ is simply connected for some (and hence all) $\beta_0 \in \Phi(M).$
\item Each fiber of $f$ contains at most one $T$-orbit of non-degenerate singular points with a hyperbolic block.
\end{itemize}
Then the fiber $f^{-1}(\beta,c)$ is connected for all $(\beta,c) \in \ft^* \times \R$ if and only if no tall singular point in $M$ has a hyperbolic block and connected $T$-stabilizer.
\end{Theorem}

\begin{Remark}
    If $M$ is simply connected, then the first bullet point in Theorem \ref{thm:converse} automatically holds. To see this, first note that the quotient space $M/T$ is simply connected by, e.g., \cite[Chapter II, Corollary 6.3]{B}. The claim then follows by \cite[Theorem 1.5]{Li2}. If $M$ has fixed points then, conversely, under the assumptions of Theorem \ref{thm:equivalence} the manifold
    $M$ is simply connected by
 \cite[Theorems 1.4 and 1.5]{Li2}.
 \end{Remark}

\begin{Remark}
More generally, Theorems \ref{thm:equivalence} and \ref{thm:converse} both hold if we replace   ``$\Phi$ is proper" by ``$\Phi$ is proper as a map to an open convex subset of $\ft^*$".
To see this, simply replace Theorem~\ref{thm:connected} in the proofs by its generalization \cite[Theorem 4.3]{LMTW}. Similarly, if we replace   ``$\Phi$ is proper" by  ``$f$ is proper and the fibers of $\Phi$ are connected'', then the first claim in Theorem \ref{thm:equivalence} still holds; see Proposition \ref{prop:extra-result}.
\end{Remark}

The fundamental idea behind the proofs of
Theorems \ref{thm:equivalence} and \ref{thm:converse} is quite simple:
Since $\left(M,\omega, \ft \times \R, f=(\Phi,g)\right)$ is an integrable system,  the function  $g$ is $T$-invariant, and hence induces a function $\overline g \colon \Phi^{-1}(\beta)/T \to \R$.
The  key technical result in our paper is that we can choose a smooth structure on the reduced space $\Phi^{-1}(\beta)/T$ so that each non-trivial component is a smooth oriented surface and
$\overline{g} \colon \Phi^{-1}(\beta)/T \to \R$ is a Morse function; moreover,
the critical points of index $1$ correspond exactly to orbits of tall singular points with a hyperbolic block and connected $T$-stabilizer. This is stated in Proposition~\ref{prop:Morse}, which combines Propositions \ref{connected}, \ref{gcrit} and \ref{nondegenerate}.
Since $\Phi$ is proper, the reduced space $\Phi^{-1}(\beta)/T$ is compact and connected  \cite{At,GS2,LMTW}.  Therefore, the proofs reduce to  standard results in Morse theory; see Lemmas~\ref{Morse2} and \ref{even}.

This method can be used more generally on integrable systems that extend a complexity one torus action but have (sufficiently nice) {\em degenerate} singular points.  For example, in a forthcoming paper \cite{ST}, we prove that the analogs of Theorems \ref{thm:equivalence} and \ref{thm:converse} hold if we also allow systems with a broader class of  ``ephemeral" singular points, including more general  resonances than that in Example~\ref{exm:short-degenerate}. Like focus-focus points and hyperbolic points with disconnected $T$-stabilizers, these ephemeral singular points correspond to regular points on the symplectic quotients.

\begin{Remark}\label{highergenus}
Let $\left(M,\omega, \ft \times \R, f:=(\Phi,g)\right)$ be an integrable system such that $(M,\omega,\Phi)$ is a complexity one $T$-space with a proper moment map. Assume that each tall singular point is non-degenerate. However, suppose that the reduced space $\Phi^{-1}(\beta)/T$ has positive genus $k$. (For example, let $M$ be the product of a symplectic toric manifold with a closed, orientable surface of genus $k$, the map $\Phi$ be (the pullback of) the moment map for the toric action, and $g$ be (the pullback of) a Morse function on the surface.) In this case, by a variation on the above argument, $\Phi^{-1}(\beta)/T$ has at least $2 k$ orbits of non-degenerate singular points with a hyperbolic block and connected $T$-stabilizer. Additionally, if $\Phi^{-1}(\beta)/T$ has {\em more} than
 $2k$ such orbits, then $\overline{g}$ must have additional local minima or maxima;
 hence the fibers of $f$ are not connected by Morse theory; see Lemma \ref{uniquemin}.
\end{Remark}

\begin{Remark}
    The second bullet point in Theorem \ref{thm:converse}  can be thought of as a  genericity condition because generic Morse functions have at most one critical point (of index 1) in each fiber.
\end{Remark}

\subsection*{Structure of the paper}
In Section~\ref{section:complexity_one}, we  review some facts about tall complexity one local models and
then briefly explore a few additional properties.
Next, in  Section~\ref{sec:key-proposition} we prove that we can endow the symplectic quotient $\Phi^{-1}(\beta)/T$ with a smooth structure so that $\overline{g}$ is a Morse function -- as long as most singular points satisfy a certain technical condition that guarantees that we can do this locally.  Most of the rest of the paper is dedicated to proving that non-degenerate singular points satisfy this condition. For example, in Section~\ref{sec:purely_elliptic_type} we prove that tall points that have purely elliptic type satisfy it, and in Section~\ref{sec:semitoric} we show that, if $M$ is four-dimensional, focus-focus points satisfy it.  At this point, we are able to demonstrate our technique by using it to give a new proof that fibers of a semi-toric system are connected. In
Section~\ref{sec:non-degenerate}, we prove that all the remaining types of non-degenerate singular points also satisfy this technical condition. Therefore, we can complete the proof of our theorems in Section~\ref{sec:proofs}.

\subsection*{Acknowledgments} D. Sepe was partially supported by FAPERJ grant JCNE E-26/202.913/2019 and by a CAPES/Alexander von Humboldt Fellowship for Experienced Researchers 88881.512955/2020-01. This study was financed in part by the Coordenação de Aperfeiçoamento de Pessoal de Nível Superior -- Brazil (CAPES) -- Finance code 001.  S. Tolman was partially supported by NSF  DMS Award 2204359 and Simons Foundation Collaboration Grant 637995.

\section{Tall local models}\label{section:complexity_one}

In this section we review some facts about tall complexity one local models.
In particular,  we introduce (tall) complexity one local models,  construct the ``defining monomial" for tall local models, and use it to describe the symplectic quotient $\YmodT$. Moreover, we  provide a generating set for $T$-invariant polynomials on these models. 

Fix a complexity one $T$-space  $(M,\omega,\Phi)$ and
let $\torbit$ be the $T$-orbit through a point $p \in M$. 
Since $\mathfrak O$ is isotropic, the vector space 
$(T_p \torbit)^{\omega}/T_p\torbit$ inherits a natural symplectic structure.
Moreover,  the natural action
of the stabilizer $H \subseteq T$ of $p$ on $T_p M$ induces a symplectic representation of $H$ on
$(T_p \torbit)^{\omega}/T_p\torbit$, called the {\bf symplectic slice representation} at $p$.
It is isomorphic to the representation of $H$ on $\C^{h+1}$
associated to some injective homomorphism  $\rho \colon H \to (S^1)^{h+1}$, where $h = \dim H$.
Let $\eta_i \in \fh^*$ be the differential of the $i$'th component of $\rho$.
The  (not necessarily distinct) vectors $\eta_0,\dots,\eta_h$  are
called the {\bf isotropy weights} at $p$; the multiset of isotropy weights at $p$ is
uniquely determined.

Set $Y := T \times_H (\fh^{\circ} \times \C^{h+1})$, where $\fh^{\circ}$ is the annihilator of $\fh$ in $\ft^*$ and $H$ acts on $T$ by multiplication, on $\fh^{\circ}$ trivially, and on $\C^{h+1}$ by $h \cdot z = \rho(h^{-1})z$.
Fix an inner product on the Lie algebra $\ft$ once and for all, and use it to identify $\ft^*$ with $\fh^\circ \oplus \fh^*.$ 
There exists a  canonical symplectic form $\omega_Y$ on $Y$ so that the $T$-action $s \cdot [t,\alpha,z] = [st,\alpha,z]$
has homogeneous moment
map $\Phi_Y \colon Y \to \ft^*$ given by 
$$\textstyle \Phi_Y([t,\alpha,z]) = \alpha + \Phi_H(z),$$
where $\Phi_H(z) = \frac{1}{2} \sum_i   |z_i|^2 \eta_i$ is the homogeneous moment
map for the $H$-action on $\C^{h+1}$; see Remark \ref{rmk:symp-form}. By the Marle-Guillemin-Sternberg local normal form theorem, there is
a $T$-equivariant symplectomorphism from an invariant neighborhood of $\torbit$
to an open subset of  $Y$  taking $p$ to $[1,0,0]$.

\begin{Definition} \label{symplectic_form_model} The triple $(Y,\omega_Y,\Phi_Y)$ is the {\bf local model} for $p$. 
\end{Definition}

\begin{Remark}\label{rmk:symp-form}
 To construct the symplectic form on $Y$, let $i^* \colon \ft^* \to \fh^*$ be the projection induced by the inclusion $H \hookrightarrow T$. Equip $T^* T \times \C^{h+1}$ with the  symplectic form  given by the sum of the pullbacks of the canonical symplectic form on $T^*T $ and the standard symplectic form on $\C^{h+1}$. The $H$-action on $T^*T \times \C^{h+1} \cong T \times \ft^* \times \C^{h+1}$ given by $h \cdot (t,\beta,w) := (ht, \beta, \rho(t^{-1})w)$ has a moment map $T \times \ft^* \times \C^{h+1} \to \fh^*$ that takes $(t,\beta,z)$ to $i^*(\beta) - \Phi_H(z)$.  The reduced space at $0$ is smooth, and the map $Y \to (T \times \ft^* \times \C^{h+1})/\!/H$ taking $[t,\alpha,z]$ to $[t,\alpha+\Phi_H(z),z]$ is a diffeomorphism. By definition, the symplectic form $\omega_Y$ on $Y$ is the pullback of the reduced symplectic form along this diffeomorphism.
\end{Remark}

We say
that the local model $Y$ is {\bf short} (respectively {\bf tall}) if $[1,0,0] \in Y$ is short (respectively tall). As shown in \cite{KT1}, tall local models\footnote{Here, and throughout the paper, all the local models that we consider are complexity one.} have an important invariant, called the {\em defining monomial}.

\begin{Lemma}\label{lemma:defn_poly}
 Let $Y = T \times_H (\fh^{\circ} \times \C^{h+1})$ be a local model associated to an injective homomorphism $\rho \colon H \to (S^1)^{h+1}$. Then there exists $\xi = (\xi_0,\ldots, \xi_h) \in \Z^{h+1}$ such that the homomorphism $(S^1)^{h+1} \to S^1$
given by $\lambda \mapsto
\prod\limits_{j=0}^h \lambda_j^{\xi_j}$ induces a short exact sequence
\begin{equation}
  \label{eq:3}
  1 \to H \stackrel{\rho}{\to} (S^1)^{h+1} \to S^1 \to 1;
\end{equation}
moreover, the vector $\xi$ is unique up to sign. Finally,  $\xi_i\xi_j \geq 0$ for all $i,j$ if and only if the local model $Y$ is tall.
\end{Lemma}

\begin{proof}
Existence of $\xi$ and its uniqueness up to sign follow from \cite[Lemma 5.8]{KT1}. By \cite[Lemma 5.4]{KT1} (and its proof), the local model $Y$ is tall  exactly if $\Phi_Y$ is not proper, which, by \cite[Lemma 5.8]{KT1}, is equivalent to $\xi_i\xi_j \geq 0$  for all $i,j$.
\end{proof}

\begin{Remark}\label{rmk:brilliant}
Let $\eta_0,\ldots, \eta_h \in \h^*$ be the isotropy weights at $p = [1,0,0] \in Y$.
By \eqref{eq:3}, the homomorphism $\R^{h+1} \to \R$ given by $a
  \mapsto \sum\limits_j a_j\xi_j$ induces dual short exact sequences
  \begin{gather*}
    0 \to \fh \stackrel{\rho_*}{\to} \R^{h+1} \to \R \to 0; \\
    0 \to \R^* \to \left(\R^{h+1}\right)^* \stackrel{\rho^*}\to \fh^* \to 0.
    \end{gather*}
  \noindent
  In particular, $b = (b_0,\ldots,b_h) \in \R^{h+1}$ lies in $\rho_*(\fh)$ exactly if $\sum\limits_j b_j \xi_j = 0$,  while $\rho^*(a) = \sum\limits_j a_j \eta_j$
  vanishes exactly if $a$ is a multiple of $\xi$.
  Hence, if $Y$ is tall then $(1,\dots,1) \notin \rho_*(\fh)$ because $\xi \neq 0$ can be chosen to lie in $\Z_{\geq 0}^{h+1}.$
  
\end{Remark}

\begin{Definition}\label{defn:defining_poly}
 Fix a tall local model $Y$. By Lemma~\ref{lemma:defn_poly}, there exists a unique $\xi = (\xi_0,\ldots, \xi_h) \in \Z_{\geq 0} ^{h+1}$ inducing the short exact sequence \eqref{eq:3}. The map $P \colon \C^{h+1} \to \C$ given by $P(z) = \prod_{j=0}^hz_j^{\xi_j}$ is called the {\bf defining monomial} of $Y$. By a slight abuse of notation, we  use the same name for the map $P\colon Y \to \C$ given by $[t,\alpha,z]\mapsto
  \prod_{j=0}^hz_j^{\xi_j}$. The sum  $N := \sum_{j=0}^h \xi_j$  is called the {\bf degree} of $P$.
\end{Definition}

\begin{Example}\label{ex1}
Let $\Z_2 = \{\pm 1\}$ act on $\C$ by $\lambda \cdot z = \lambda z$. The associated local model  $S^1 \times_{\Z_2} \R^* \times \C$ is tall.
The defining monomial $P\colon \C \to \C$ is given by $P(z) = z^2$ and has degree $2$.
\end{Example}

\begin{Example} \label{ex2}
Let $S^1$ act effectively on $\C^2$ by 
$\lambda \cdot (z_1,z_2) = (\lambda z_1, \lambda^{-1} z_2)$ with moment map $\Phi (z_1,z_2)= \frac{1}{2}(|z_1|^2 - |z_2|^2)$. This is a tall local model.
%
The defining monomial $P \colon \C^2 \to \C$ is given by $P(z_1,z_2) = z_1 z_2$ and also has degree $2$.
\end{Example}

As in \cite{KT1}, we say that a point $p \in M$ (and the orbit $[p] \in M/T$) is {\bf exceptional} if  every point in some neighborhood  of
$[p]$ in the reduced space $\Phi^{-1}(\Phi(p))/T$  has strictly smaller stabilizer.  By definition, every short point is exceptional. The following criterion  determines if a point in a tall local model is exceptional.


\begin{Lemma} \label{lemma:criterion-exc} Fix a tall local model $Y= T \times_H (\fh^\circ \times \C^{h+1})$ with defining monomial $P(z) = \prod_{i=0}^h z_i^{\xi_i}$. Then $ [ t, \beta, z ] \in Y$ is exceptional exactly if $\sum_{z_i = 0} \xi_i > 1$, where the sum is over all $i$ such that $z_i = 0$.
 \end{Lemma}

 \begin{proof} 
 Arguing as in the proof of Lemma \ref{lemma:defn_poly}, since $Y$ is tall the moment map $\Phi_Y$ is not proper. 
 Therefore, the claim can be obtained by joining the statements of \cite[Lemmas 5.15 and 5.17]{KT1}. 
 \end{proof}

The reduced space $\YmodT = \Phi^{-1}_Y(0)/T$ is endowed with the subquotient topology and a sheaf of {\bf smooth functions} defined as follows:  For any open subset $U \subseteq \YmodT$, let $C^\infty(U)$ be the set of functions $\overline{g} \colon U \to \R$ whose pullback to $\Phi_Y^{-1}(0)$ is the restriction of a
smooth $T$-invariant function $g \colon Y \to \R$. A homeomorphism between topological spaces endowed with sheaves of smooth functions is a {\bf diffeomorphism} if it induces an isomorphism between the sheaves of smooth functions.


\begin{Lemma}\label{trivial}
  Let $Y = T \times_H (\fh^\circ \times \C^{h+1})$ be a tall local model
  with moment map $\Phi_Y \colon Y \to \ft^*$. 
 The defining monomial $P \colon Y \to \C$ induces a homeomorphism $ \overline{P} \colon \YmodT \to \C$ that restricts to a diffeomorphism on  the set 
 of non-exceptional orbits. Moreover, $[q] \in \YmodT$ is exceptional exactly if $[q] = [1,0,0]$ and the degree $N$ of the defining polynomial of $Y$ is greater than $1$.
\end{Lemma}

\begin{proof}
 Arguing as in the proof of Lemma \ref{lemma:defn_poly}, since $Y$ is tall the moment map $\Phi_Y$ is not proper. Hence, by \cite[Corollary 6.3]{KT1} the map $\overline{P} \colon \YmodT \to \C$ is a homeomorphism. Moreover, by \cite[Corollary 7.2]{KT1} the restriction of $\overline{P} \colon \YmodT \to \C$ to the complement of the exceptional orbits is a diffeomorphism onto its image.  By Lemma \ref{lemma:criterion-exc},  if $q \in \Phi_Y^{-1}(0)$ is exceptional then $\overline{P}([q]) = 0$.  Since $P([1,0,0]) = 0$, this implies that $[1,0,0]$ is the only possible exceptional orbit.  Finally, by Lemma~\ref{lemma:criterion-exc}, $[1,0,0]$ is exceptional exactly if $N > 1$.  
\end{proof}

The next result allows us to understand the smooth structure on the reduced
space $\YmodT$ even when the homeomorphism $\overline{P} \colon \YmodT \to \C$ is not a diffeomorphism.

\begin{Lemma}\label{lemma:hom}
Let $Y = T \times_H (\fh^\circ \times \C^{h+1})$ be a tall local model with defining monomial $P(z) = \prod_{i=0}^h z_i^{\xi_i}$.  Assume that $\xi_i = 0$ for all $i > k$.
Let the group $$\textstyle K := \{ \lambda \in (S^1)^{k+1} \mid \prod_{j=0}^k \lambda_j^{\xi_j} = 1 \}$$  act naturally on  $\C^{k+1}$.
The inclusion $\C^{k+1} \hookrightarrow Y$ given by $$(z_0,\ldots, z_{k}) \mapsto [1,0,(z_0,\ldots, z_{k}, 0,\ldots, 0)]$$
induces a diffeomorphism between $\C^{k+1}/\!/K$ and $\YmodT$.
\end{Lemma}

\begin{proof} Let $H$ act on $\C^{h+1}$ via $\rho \colon H \hookrightarrow (S^1)^{h+1}$. 
Let  $\Phi_H \colon \C^{h+1} \to \fh^*$ 
and $\Phi_K \colon \C^{k+1} \to \mathfrak{k}^*$ be the homogeneous moment maps for the $H$-action on $\C^{h+1}$ and the $K$-action on $\C^{k+1}$, respectively.
Since $\xi_i = 0$ for all $i > k$, we have
$\Phi_H^{-1}(0) = \Phi_K^{-1}(0) \times \{0\}$, and so $\Phi_Y^{-1}(0) =  
T \times_H (\{0\} \times (\Phi_K^{-1}(0) \times \{0\}))$, where $\Phi_Y : Y \to \ft^*$ is the moment map for the $T$-action on $Y$.
Since $\rho(H)= K \times (S^1)^{h-k}$, the claim follows immediately.
\end{proof}

On the reduced space $\YmodT$ there is a simple relation between the magnitudes of the defining monomial and of $z$.
\begin{Lemma}\label{lemma:constant}
Let $Y := T \times_H (\fh^\circ \times \C^{h+1})$ be a tall 
local model with moment map $\Phi_Y$.  
Let $N$ be the degree of the
defining monomial  $P(z) = \prod_{j} z_j^{\xi_j}$.
Then for all $[t,\alpha,z] \in \YmodT$,
$$\textstyle \big\lvert z \rvert^2 =N \big(\prod_j \xi_j^{\xi_j}\big)^{-\frac{1}{N}}\lvert P(z) \big\rvert^\frac{2}{N},$$
where $0^0 = 1$ by convention.
\end{Lemma}
\begin{proof}
Let $\eta_0,\dots,\eta_h \in \mathfrak h^*$ be the weights for the $H$-action on $\C^{h+1}$.
The moment map $\Phi_Y$ sends $[t,\alpha,z] \in Y$ to  $\alpha + \Phi_H(z) \in \ft^*$,
where $\Phi_H(z) = \sum_{j=0}^h \frac{1}{2} |z_j|^2 \eta_j$.
Thus it suffices to check that $\lvert z  \rvert^2 = \big(\prod_j \xi_j^{\xi_j}\big)^{-\frac{1}{N}} \big\lvert P(z) \big\rvert^\frac{2}{N}$
for all $z \in \Phi_H^{-1}(0)$.
Fix $z \in \Phi_H^{-1}(0)$.  
Since the local model $Y$ is tall and $\Phi_H(z) = \sum_{j=0}^h \frac{1}{2} |z_j|^2 \eta_j = 0$, by Remark \ref{rmk:brilliant} there exists $c \geq 0$ 
such that $|z_j|^2 = c \xi_j$ for all $j$.
Therefore,
\begin{equation*}
  \begin{split}
    \lvert z \rvert^2 &= \sum\limits_{j=0}^{h} \lvert z_j \rvert^2 =
    cN; \\
    |P(z)|^2 = \Big \lvert  \prod_{j=0}^{h} z_j^{\xi_j} \Big \rvert^2 &= \prod\limits_{j=0}^{h} \lvert z_j
    \rvert^{2\xi_j} =
    c^N \prod\limits_{j=0}^{h}\xi_j^{\xi_j}.
  \end{split}
\end{equation*}
\noindent
The result holds by combining these equations.
\end{proof}

To study integrable systems that extend complexity one $T$-actions, we need to understand $T$-invariant functions on these spaces.  We start by considering the simplest $T$-invariant functions on local models: $T$-invariant polynomials.

\begin{Definition}\label{t-invariantpoly}
  Let $Y = T \times_H (\fh^{\circ} \times \C^{h+1})$ be a local model. By a slight abuse of notation, we call a $T$-invariant function
  $g \colon Y \to \R$ a {\bf $\boldsymbol{T}$-invariant (homogeneous) polynomial} if its restriction to $\fh^\circ \times \C^{h+1}$ is a (homogeneous) polynomial.
\end{Definition}

\begin{Lemma}\label{invariant_polys}
Let $Y = T \times_H (\fh^{\circ} \times \C^{h+1})$ be a tall 
local model with moment map $\Phi_Y \colon Y \to \ft^*$.  
The algebra of $T$-invariant polynomials on $Y$
is generated by the real and imaginary parts of the defining monomial $P \colon Y \to \C$, the map $[t,\alpha,z] \mapsto |z|^2$, and the components of  $\Phi_Y$.
\end{Lemma}

\begin{proof}
   By definition, there is a one-to-one correspondence between $T$-invariant polynomials on $Y$ and
   $H$-invariant polynomials on $\fh^\circ \times \C^{h+1}$. Let $H$ act on $\C^{h+1}$ via $\rho \colon H \hookrightarrow (S^1)^{h+1}$ and let $\Phi_H \colon \C^{h+1} \to \fh^*$ be the homogeneous moment map. Since  $\Phi_Y([t,\alpha,z]) = \alpha + \Phi_H(z)$ for all $[t,\alpha, z] \in
   Y$, it suffices to
   prove that the algebra of $H$-invariant polynomials on $\C^{h+1}$ is generated by the real and imaginary parts of the
   defining monomial $P \colon \C^{h+1} \to \C$, the map $z \mapsto |z|^2$, and the components of $\Phi_H$. Since the local model $Y$ is tall, Remark \ref{rmk:brilliant} implies that $(1,\ldots,1) \notin \rho_*(\fh)$. Hence, since $\Phi_H(z) = \frac{1}{2}\rho^*(|z_0|^2,\ldots, |z_h|^2)$, 
   the subalgebra of $H$-invariant polynomials on $\C^{h+1}$ generated by the map
   $z \mapsto |z|^2$ and the components of $\Phi_H$ is also generated by the functions $z \mapsto
     |z_i|^2$ for  $i =0,\ldots, h$. 

  Let  $P \colon \C^{h+1} \to \C$ be given by $P(z) = \prod_i z_i^{\xi_i}$.
  Consider a complex-valued polynomial of the form $Q(z) = \prod_i z_i^{j_i} \, \overline {z}_i^{k_i}$.
  Given $\lambda \in (S^1)^{h+1}$, $Q(\lambda z) = \prod_i \lambda_i^{j_i - k_i} Q(z)$.
  In particular, $Q$ is $H$-invariant exactly if $ \prod_i \lambda_i^{j_i-k_i} = 1$ for all $\lambda \in \rho(H) \varsubsetneq (S^1)^{h+1}$. Hence, by
  Lemma \ref{lemma:defn_poly}, $Q$ is $H$-invariant exactly if the
  vector $(j_0-k_0,j_1-k_1,\dots,j_h-k_h)$ is a multiple of $(\xi_0,\dots,\xi_h)$, i.e., 
  exactly if $Q$ is some product of $P$, its conjugate, and polynomials  of the form $z \mapsto |z_i|^2$. Since the the vector space of complex-valued polynomials on $\C^{h+1}$ decomposes as the direct sum of weight spaces that are spanned by elements of the form $z \mapsto \prod_i z_i^{j_i}\overline {z}_i^{k_i}$, it follows that $H$-invariant complex-valued
  polynomials are generated by $P$, its conjugate,
  and the polynomials $z \mapsto |z_i|^2$. The claim follows immediately.
\end{proof}

\begin{Remark}\label{rmk:other_basis}
  The proof of Lemma \ref{invariant_polys} provides an equivalent generating set for the algebra of $T$-invariant polynomials on a tall local model $Y$, namely the real and imaginary parts of the defining monomial $P \colon Y \to \C$, the maps $|z_i|^2 \colon Y \to \R$ that send $[t,\alpha,z]$ to $|z_i|^2$ for $i=0,\ldots, h$, and the components of the map $[t,\alpha,z] \mapsto \alpha$.
\end{Remark}

\section{The key proposition}\label{sec:key-proposition}
In this section, we prove that we can endow each non-trivial symplectic quotient  with the  structure  of a smooth oriented surface
so that the induced function $\overline{g}$ is Morse -- as long as most singular points satisfy a certain technical condition (see Proposition \ref{connected}).  
This eventually allows us to use Morse theory in two dimensions to obtain the desired results.

Let $\left(M,\omega, \ft \times \R, f:=\left(\Phi,g\right)\right)$ be an integrable
    system such that $\left(M,\omega,\Phi\right)$ is a complexity
  one $T$-space. As in the Introduction,  we denote the induced function on the reduced space $\Phi^{-1}(\Phi(p))/T$ by $\overline{g}$.
We need the following definition.

\begin{Definition}\label{critmodphi} 
  A point $p \in M$ is a {\bf critical point of $\boldsymbol{g}$ modulo $\boldsymbol{\Phi}$} if
  $D_p g$ vanishes on $\ker D_p\Phi$; otherwise, $p$ is a {\bf regular point of $\boldsymbol{g}$ modulo $\boldsymbol{\Phi}$.}
\end{Definition}
 Clearly, every critical point of $g$ modulo $\Phi$ is a singular point of $(M,\omega,\ft \times \R, f:=(\Phi,g))$. If $\Phi$ is proper then the converse holds over the interior of $\Phi(M)$; see Remark~\ref{rmk:crit} below.
 
 \begin{Remark}
    By linear algebra,  $(\ker D_p \Phi)^{\circ} = \{D_p\Phi^{\xi} \mid \xi \in \ft\}$.  Hence, $p$ is a so-called ``relative equilibrium of the Hamiltonian vector field $X_g$ with respect to the symmetry given by the $T$-action," 
i.e.,  there exists $\xi \in \ft$ such that $D_p g = D_p \Phi^{\xi}$,
exactly if $D_p g \in (\ker D_p \Phi)^{\circ}$,  which occurs  exactly
if $p$ is a critical point of $g$ modulo $\Phi$.
\end{Remark}

To understand the next lemma, recall that by Lemma~\ref{trivial}  there is a natural structure of smooth manifold on the reduced spacenear each non-exceptional orbit, with respect to which the induced function $\overline{g} \colon \MmodT \to \R$ is smooth.

\begin{Lemma}\label{criticalcritical}
    Let $\left(M,\omega, \ft \times \R, f:=\left(\Phi,g\right)\right)$ be an integrable
    system such that $\left(M,\omega,\Phi\right)$ is a complexity
  one $T$-space. If $p \in \Phi^{-1}(0)$ is not exceptional, then $[p]$ is a critical point of $\overline{g} \colon \MmodT \to \R$ exactly if $p$ is critical point of $g$ modulo $\Phi$. \end{Lemma}
\begin{proof}
 By the Marle-Guillemin-Sternberg local normal form theorem, we may assume that $M$ is a local  model $Y = T \times_H (\fh^\circ \times \C^{h+1})$, that $\Phi = \Phi_Y$, and that $p =
[1,0,0]$. Since $p$ is not exceptional, $p$ is tall. Let $P(z) = \prod_{i=0}^h z_i^{\xi_i}$ be the defining monomial of $Y$.  Since $\sum_i \xi_i = 1$ by Lemma \ref{lemma:criterion-exc},  we may further assume that $\C^{h+1} \simeq \C \times \C^{h}$ and $H \simeq (S^1)^h$, where $(S^1)^h$ acts on $\C^{h}$ in the standard fashion and trivially on $\C$. 

We identify $T_p Y$ with $\fh^{\perp} \times \fh^\circ \times \C^{h+1}$ via the map $\fh^{\perp} \times \fh^{\circ} \times \C^{h+1} \to Y$ given by $(\xi, \alpha,z) \mapsto [\exp(\xi),\alpha,z]$. The kernel of $D_p \Phi$ is $ \fh^{\perp}\times \{0\} \times \C^{h+1}$. Moreover, since $g$ is $T$-invariant and $H$ acts on $\C^h$ via an isomorphism with $(S^1)^h$, the functional $D_p g$ is $H$-invariant and vanishes on $\{0\} \times \{0\} \times (\{0\} \times \C^h)$. Hence, $p$ is a critical point of $g$ modulo $\Phi$ exactly if $D_p g$ vanishes on $\fh^{\perp} \times \{0\} \times  (\C \times \{0\}) $.
Since $\Phi^{-1}(0) = T \times_H(\{0\} \times (\C \times \{0\})$,
this holds exactly if $[p]$ is a critical point of $\overline{g}$.
\end{proof}

There is also a natural structure of a smooth manifold on the reduced space near each regular point of $g$ modulo $\Phi$, but in this case the induced function $\overline{g}$ is locally regular. More precisely, the following holds.

\begin{Lemma}\label{greg2}
Let $\left(M,\omega, \ft \times \R, f:=\left(\Phi,g\right)\right)$ be an integrable
    system such that $\left(M,\omega,\Phi\right)$ is a complexity
  one $T$-space. Let $p \in M$ be a regular point of $g$ modulo $\Phi$. Then there exists an open neighborhood $U_p$ of $[p]$ in the reduced space
$\Phi^{-1}(\Phi(p))/T$ and a map $\Psi_p \colon U_p \to \R^2$ satisfying:
\begin{enumerate}[label=(\alph*),ref=(\alph*)]
    \item \label{item:-4} $\Psi_p \colon U_p \to \Psi(U_p)$ is a diffeomorphism onto an open set; and
    \item\label{item:-2} $\overline{g} \circ \Psi_p^{-1} \colon \Psi_p(U_p) \to \R$ is a smooth function with no critical points.
\end{enumerate}
Moreover, $p$ is tall, has purely elliptic type, and the defining monomial of the local model for $p$ has degree $1$. \end{Lemma}

\begin{proof}
By the Marle-Guillemin-Sternberg local normal form theorem, we may assume that $M$ is a local  model $Y = T \times_H (\fh^\circ \times \C^{h+1})$, that $\Phi = \Phi_Y$, and that $p =
[1,0,0]$. First we show that $Y$ is tall. If it is short, then $\Phi^{-1}(\Phi(p)) = T \cdot p$. Since $\Phi$ is invariant under the Hamiltonian flow of $g$, this implies that $(T \times \R) \cdot p = T \cdot p$. Hence $p$ is a critical point of $g$ modulo $\Phi$, which contradicts our assumption.

We identify $T_p Y$ with $\fh^{\perp} \times \fh^\circ \times \C^{h+1}$ via the map $\fh^{\perp} \times \fh^{\circ} \times \C^{h+1} \to Y$ given by $(\xi, \alpha,z) \mapsto [\exp(\xi),\alpha,z]$. The kernel of $D_p \Phi$ is given by $ \fh^{\perp}\times \{0\} \times \C^{h+1}$.
By Remark \ref{rmk:other_basis}, if the degree $N$ of the defining monomial $P$ is greater than 1, then the only $T$-invariant polynomials of degree 1 on $Y$ are the components of the map $[t,\alpha,z]\mapsto \alpha$, while if $N=1$, the real and imaginary parts of $P$ are also $T$-invariant polynomials of degree 1. Since the derivative of the map $[t,\alpha,z]\mapsto \alpha$ at $p$ vanishes when restricted to $\ker D_p \Phi$, this implies  that $N=1$. If we assume,   without loss of generality,  that $P(z) = z_0$,
it also implies  that $D_p g$ restricts to a non-zero $H$-invariant linear function on $\{0\} \times \{0\} \times (\C \times \{0\})$. By Lemmas \ref{trivial}, $\overline{P}$ is a diffeomorphism; its inverse is given by $z \mapsto [1,0,(z,0)]$. Therefore $\overline{g} \circ \overline{P}^{-1} \colon \C \simeq \R^2 \to \R$ is equal to 
the restriction of $g$ to $\C \simeq \{1\} \times_H (\{0\} \times (\C \times \{0\})) \subseteq Y$. Hence, since the restriction of $D_p g$ to $\{0\} \times \{0\} \times (\C \times \{0\})$ is non-zero, there exists an open neighborhood $U_p$ of $[p]$ in $\YmodT$ such that the restriction of $\overline{g} \circ \overline{P}^{-1}$ to $\overline{P}(U_p)$ is smooth and has no critical points. Setting $\Psi_p \colon U_p \to \R^2$ to be the restriction of $\overline{P}$ to $U_p$, the first result follows.

Let $\mathcal{O}$ be the $(T \times \R)$-orbit through $p$. Since $D_p g$ restricts to a non-zero  function on $\C \times \{0\}$,  the vector space $\left(T_p \rorbit\right)^{\omega}/T_p\rorbit$ is isomorphic to $\C^h$ and the Lie algebra of the $(T \times \R)$-stabilizer of $p$ is $\fh \times \{0\} \varsubsetneq \ft \times \R$. Hence, the  linearized action on $\left(T_p \rorbit\right)^{\omega}/T_p\rorbit$ is isomorphic to the $H$-action on $\C^h$.
Finally, since $P(z) = z_0$,   $H$ acts on $\C^{h+1} \simeq \C \times \C^h$ through an isomorphism $H \simeq \{1\} \times (S^1)^h$.
Therefore, $p$ has purely elliptic type.
\end{proof}

\begin{Remark}\label{rmk:crit}
If $\Phi$ is proper, then every singular point of $(M,\omega,\ft \times \R, f:= (\Phi,g))$
over the interior of $\Phi(M)$ is a critical point of $g$ modulo $\Phi$. To see this, let $p$ be a regular point of $g$ modulo $\Phi$. By Lemma~\ref{greg2}, $p$ is tall and the degree of the defining monomial of the local model for $p$ is $1$.
Moreover, if $\Phi(p)$ is in the interior of $\Phi(M)$, then  the local model for $p$ has a surjective moment map because  $\Phi$ is open
 as a map to $\Phi(M)$ by Theorem~\ref{thm:connected}. Together, these facts imply that the $T$-stabilizer of $p$ is trivial.
Therefore, $p$ is a regular point of $\Phi$, and so also of $f$.
\end{Remark}

By Lemma \ref{greg2}, to construct the desired Morse function on a reduced space we need to impose local conditions near each critical point of $g$ modulo $\Phi$. More precisely, the following key proposition holds.


\begin{Proposition}\label{connected}
Let $\left(M,\omega, \ft \times \R, f:=\left(\Phi,g\right)\right)$ be an integrable system such that $\left(M,\omega,\Phi\right)$ is a complexity
  one $T$-space.  Fix $\beta \in \ft^*$.
Assume that for each critical point $p \in \Phi^{-1}(\beta)$ of $g$ modulo $\Phi$  
 there exists an open neighborhood $U_p$ of $[p]$ in  the reduced space $\Phi^{-1}(\beta)/T$  
and a map $\Psi_p \colon U_p \to \R^2$ satisfying:
\begin{enumerate}
\item $\Psi_p \colon U_p \to \Psi_p(U_p)$ is a homeomorphism onto an open set;
\item  $\Psi_p$ restricts to a diffeomorphism on $U_p \smallsetminus \{[p]\}$; and
\item $ \overline{g}\circ \Psi_p^{-1} \colon \Psi_p(U_p) \to \R$ is a Morse function.
\end{enumerate}
Then  $\Phi^{-1}(\beta)/T$ can be given the structure of a smooth oriented surface such that $\overline g \colon \Phi^{-1}(\beta)/T \to \R$ is a Morse function.
Moreover, $[p]$ is a critical point of $\overline{g}$ of index $\mu$ exactly if both $p$ is a critical point of $g$ modulo $\Phi$ and $\Psi_p([p])$ is a critical point of $\overline g \circ \Psi_p^{-1}$ of index $\mu$.
\end{Proposition}

\begin{proof}
Assume for simplicity that $\beta = 0$. By 
hypothesis and Lemma \ref{greg2}, for each $p \in \Phi^{-1}(0)$ there exists an open neighborhood $U_p$ of $[p]$ in the reduced space $\MmodT:=\Phi^{-1}(0)/T$
and a map $\Psi_p \colon U_p \to \R^2$ satisfying properties (1) -- (3) above;  
indeed, if $p$ is a regular point of $g$ modulo $\Phi$, then the stronger properties (a) and (b) of Lemma 3.3 also hold. In particular, $\Psi_p$ is a diffeomorphism on the subset of orbits of regular points of $g$ modulo $\Phi$. By shrinking $U_p$ if necessary, we may further assume that $\overline{g} \circ \Psi_p^{-1}$ is a regular function on $\Psi_p(U_p \smallsetminus \{[p]\})$.
Moreover, by the local normal form and Lemma \ref{trivial}, the set of exceptional orbits in $\MmodT$ is discrete. Therefore, by further shrinking $U_p$ if necessary,
we may assume that $U_p \smallsetminus \{[p]\}$ contains no exceptional orbits.
By Lemma~\ref{criticalcritical}, this implies that $U_p \smallsetminus \{[p]\}$ contains no critical points of $g$ modulo $\Phi$.

Hence we may choose a subset $\mathcal P \varsubsetneq \Phi^{-1}(0)$ such that $\{U_p\}_{p \in \mathcal P}$ is
a cover of $\MmodT$ and each orbit of critical points of $g$ modulo $\Phi$ is contained in exactly one open set. Therefore for all $p,p'\in \mathcal{P}$,
the transition function
$\Psi_{p} \circ \Psi_{p'}^{-1}$ is a diffeomorphism from 
$\Psi_{p'}( U_{p} \cap U_{p'})$ to $\Psi_{p}(U_{p} \cap U_{p'})$. 
Hence, the coordinate charts $\Psi_{p}$
endow the reduced space $\MmodT$ with the structure of a $2$-dimensional smooth manifold.
Moreover, with respect to this smooth structure,
$\overline{g}  \colon \MmodT \to \R$
is a Morse function. Additionally, $[p]$ is a critical point of $\overline{g}$ exactly if both $p$ is a critical point of $g$ modulo $\Phi$ and $\Psi_p([p])$ is a critical point of $\overline g \circ \Psi_p^{-1}$; in this case, the index of $[p]$ is the
index of $\Psi_p([p])$.
Finally, since every point in $\Phi^{-1}(0)$ is tall, $\MmodT$ is oriented by \cite[Proposition 6.1]{KT1}.
\end{proof}

The aim of Sections \ref{sec:purely_elliptic_type} and \ref{sec:non-degenerate} is to show that the hypotheses of Proposition \ref{connected} hold for non-degenerate singular points.

\section{Purely elliptic type}\label{sec:purely_elliptic_type}

Our goal in this section is to prove that tall 
critical points of $g$ modulo $\Phi$ that have purely elliptic type satisfy the hypotheses of Proposition \ref{connected} (see Proposition \ref{gcrit} below). An important ingredient is the local classification of such points (see Lemma \ref{need}). We begin by giving an important example.

\begin{Example}\label{exm:elliptic}
Fix $\ell \in \{0,1,\dots,n\}.$
Let $\omega_0 \in \Omega^2(  T^* (S^1)^{n-\ell} \times \C^\ell)$ be the sum of the pullbacks
of the canonical symplectic form on $T^*(S^1)^{n-\ell}$ and the standard symplectic form on $\C^{\ell}$.
Let $\Phi_0 \colon T^*(S^1)^{n - \ell} \times \C^\ell \simeq (S^1)^{n-\ell} \times \left(\R^{n-\ell}\right)^* \times \C^{\ell} \to (\R^n)^*$ be the moment map for the standard $(S^1)^n = (S^1)^{n-\ell} \times (S^1)^\ell$-action
on  $T^*(S^1)^{n - \ell} \times \C^\ell$ given by $\Phi_0(s,a,w)= (a,\frac{1}{2}|w|^2)$, where here $\frac{1}{2}|w|^2 := \frac{1}{2}(|w_1|^2,\ldots,  |w_{\ell}|^2)$. Then $(T^* (S^1)^{n - \ell} \times \C^\ell, \omega_0, \R^n , \Phi_0)$ is an integrable system; moreover,
$p = (s,a,w)$ is singular exactly if $\prod_i w_i = 0$, in which case $p$ has purely elliptic type.
\end{Example}

Next we define ``isomorphism" for the class of integrable systems that we consider. 

\begin{Definition}\label{def:isomorphism}
  Let $(M_i,\omega_i,\ft \times \R, f_i=(\Phi_i,g_i))$ be an integrable system such
  that $(M_i,\omega_i,\Phi_i)$ is a complexity one $T$-space for $i=1,2$.
  An {\bf isomorphism} from  $(M_1,\omega_1, \ft \times \R, f_1)$ to 
  $(M_2,\omega_2, \ft \times \R , f_2)$ is a symplectomorphism $\Psi\colon (M_1,\omega_1) \to (M_2,\omega_2)$ and a diffeomorphism\footnote{Let $X, Z$ be manifolds and let $A \subseteq X$. A map
    $\psi \colon A \to Z$ is {\bf smooth} if, for every point
    $a \in A$, there exists an open set $U_a \subseteq X$ and a smooth
    map $\psi_a \colon U_a \to Z$ extending $\psi$.} $\psi\colon f_1(M_1) \subseteq \ft^* \times \R \to f_2(M_2) \subseteq \ft^*
  \times \R$ of the form 
  $$ \psi(\beta, c) = (\tau(\beta),\psi_{\R}(\beta,c) )\text{ for } (\beta, c) \in f_1(M_1),$$
  where $\tau\colon \ft^* \to \ft^*$ is a translation, 
  such that $f_2 \circ \Psi = \psi \circ f_1$.
\end{Definition}

We can now give the local classification for critical points of $g$ modulo $\Phi$ that have purely elliptic type.

\begin{Lemma}\label{need}
Let $(M,\omega,\ft \times \R, f := (\Phi,g))$
be an integrable system
such that  $(M,\omega,\Phi)$ is a complexity one $T$-space.
    Let $p \in M$ be a critical point of $g$ modulo $\Phi$ that has purely elliptic type. 
 Let $Y = T \times_H (\fh^\circ \times \C^{h+1})$ be the local model for $p$, where $H$ acts on $\C^{h+1}$ via $\rho \colon H \to (S^1)^{h+1}$.
    Given $b \in \R^{h+1} \smallsetminus \rho_*(\fh)$, a neighborhood of $T \cdot p$ in $(M,\omega,\ft \times \R, f = (\Phi,g))$ is isomorphic to a neighborhood of $T \cdot [1,0,0]$
in  $(Y,\omega_Y, \ft \times \R, f_b = (\Phi_Y, g_b))$, where $g_b \colon Y \to \R$
is given by $g_b([t,\alpha,z]) = \sum_{i=0}^h |z_i|^2 b_i$.
\end{Lemma}

\begin{proof}
Since $p$ is a critical point of $g$ modulo $\Phi$, $\ker D_p f = \ker D_p  \Phi$, and so the $\left(T \times \R\right)$-orbit of $p$ is equal to the $T$-orbit of $p$. In particular, the rank of $D_p f$ is $n-h-1$ and
the $\left(T \times \R\right)$-orbit of $p$ is compact, where $h$ is the dimension of the stabilizer of $p$ for the $T$-action. By the local normal form for compact orbits of points that have purely elliptic type \cite{E}, we may assume that $M$ and the symplectic form are as in Example \ref{exm:elliptic}, where $\ell = h+1$, that $p = (1,0,0)$, and that there exists a diffeomorphism $\varphi \colon \left(\R^n\right)^*
\to \ft^* \times \R$  such that  $f = \varphi \circ \Phi_0$, where $\Phi_0$ is also as in Example \ref{exm:elliptic}.

Since $\varphi \circ \Phi_0 = f = (\Phi,g)$, there exist an
injective Lie group homomorphism $\Lambda\colon T \to (S^1)^n = (S^1)^{n-h-1} \times (S^1)^{h+1}$, $\beta \in \ft^*$,  and a smooth map $\varphi_\R \colon (\R^n)^* \to \R$ such that $\varphi = (\Lambda^* + \beta, \varphi_{\R})$, where $\Lambda^*$ is the dual of the Lie algebra homomorphism $\Lambda_* : \ft \to \R^n$. By replacing $f$ by its pre-composition with the translation by  $-\beta$, we may assume that $\beta = 0$. Hence, the composite $D_0\varphi \circ \varphi^{-1} \colon \ft^* \times \R \to \ft^* \times \R$ is the identity on $\ft^*$. Therefore, by replacing $f$ by the composition $D_0 \varphi \circ \varphi^{-1} \circ f = D_0 \varphi \circ \Phi_0$, we may assume that $\varphi_{\R}$ is linear.

Let $\Lambda_1$ and $\Lambda_2$ denote the
projections of $\Lambda$ onto $(S^1)^{n-h-1} $ and $ (S^1)^{h+1}$, and let $\Lambda_1^*$ and $\Lambda^*_2$ be the dual of the Lie algebra homomorphisms induced by $\Lambda_1$ and $\Lambda_2$ respectively. Then $\ker \, \Lambda_1 = H$ and $\Lambda_2|_H = \rho$. Hence, if $b \in \R^{h+1} \smallsetminus \rho_*(\fh)$,
then $(0,b) \in \R^n \smallsetminus \Lambda_*(\ft)$.
Thus there exists a linear isomorphism $L$ of $\ft^* \times \R$ that restricts to the identity on $\ft^*$ so that, by further replacing $f$ by the composition $L \circ f$, we may assume that $\varphi_{\R} = (0,2b)$. In particular, $g(s,a,w) = \sum_{i=0}^h b_i|w_i|^2$.

By construction, the map given by $$[t,\alpha,z] \mapsto [t,\alpha+\Phi_H(z),z]$$
\noindent
is a symplectomorphism between the local model $Y$ and the reduced space $(T \times \ft^* \times \C^{h+1})/\!/H$; see Remark \ref{rmk:symp-form}.

By a straightforward computation, the map $T \times \ft^* \times \C^{h+1} \to T \times \ft^* \times \C^{h+1}$ given by $ \textstyle (t,\beta,z) \mapsto \left(t,\beta - \Lambda^*_2 \left(\frac{1}{2}|z_0|^2,\ldots,\frac{1}{2}|z_{h}|^2\right), \Lambda_2(t)z\right)$
is a symplectomorphism. It induces a symplectomorphism between the reduced space $(T \times \ft^* \times \C^{h+1})/\!/H$ and $T/H \times \fh^{\circ} \times \C^{h+1}$ that is given by $$ \textstyle [t,\beta,z] \mapsto (tH,\beta - \Lambda^*_2 \left(\frac{1}{2}|z_0|^2,\ldots,\frac{1}{2}|z_{h}|^2\right), \Lambda_2(t)z). $$ 
\noindent
Here the symplectic form on $T/H \times \fh^{\circ} \times \C^{h+1}$ is the sum of the pullbacks of the canonical symplectic form on $T^* (T/H) \cong T/H \times \fh^{\circ}$ and of the standard symplectic form on $\C^{h+1}$.

By a dimension count, $\Lambda_1$ induces an isomorphism between $T/H$ and $(S^1)^{n-h-1}$. Hence it induces a symplectomorphism between $T/H \times \fh^{\circ} \times \C^{h+1}$ and $(S^1)^{n-h-1} \times \left(\R^{n-h-1}\right)^* \times \C^{h+1}$ that is given by 
$$(tH,\alpha,w) \mapsto (\Lambda_1(t), (\Lambda^*_1)^{-1}(\alpha),w).$$

Composing these maps, we obtain a symplectomorphism $\Psi\colon Y \to (S^1)^{n-h-1} \times \left(\R^{n-h-1}\right)^* \times \C^{h+1}$. Since $\Lambda_2|_H = \rho$, it is given by 
$$ \textstyle \Psi([t,\alpha,z]) = \left(\Lambda_1(t),(\Lambda_1^*)^{-1}\left(\alpha -    \mathrm{pr}_{\fh^\circ}\left(\Lambda^*_2\left(\frac{1}{2}|z_0|^2,\ldots,\frac{1}{2}|z_{h}|^2\right)\right)\right), \Lambda_2(t)z\right). $$
Here $\mathrm{pr}_{\fh^\circ}\colon \ft^* \to \fh^\circ$
denotes the projection induced by the fixed choice of metric on $\ft$. Clearly this map is $T$-equivariant and so $\Phi \circ \Psi$ is a moment map for the $T$-action on $Y$. Since $\Phi \circ \Psi(1,0,0) = 0$, it follows that $\Phi \circ \Psi$ is the homogeneous moment map $\Phi_Y$. Moreover, it is easy check that $g \circ \Psi = g_b$. By construction, the integrable system $(Y,\omega_Y,\ft \times \R, f\circ \Psi = (\Phi_Y, g_b))$ is isomorphic to the original one, as desired.
\end{proof}

We are finally ready to prove the main result of this section.

\begin{Proposition}\label{gcrit}
Let $(M,\omega,\ft \times \R, f := (\Phi,g))$
be an integrable system
such that  $(M,\omega,\Phi)$ is a complexity one $T$-space.
Let $p \in \Phi\inv(0) \cap g^{-1}(0)$ be a tall 
critical point
of $g$ modulo $\Phi$ that has purely elliptic type. After possibly replacing $g$ by $-g$, there exist an open neighborhood $U$ of $[p]$ in $\MmodT$
and a map $\Psi \colon U \to \R^2$ taking $[p]$ to $(0,0)$ satisfying:
\begin{enumerate}
\item $\Psi \colon U \to \Psi(U)$ is a homeomorphism onto an open set;
\item $\Psi$ restricts to a diffeomorphism on $U \smallsetminus \{[p]\}$; and
\item $(\overline{g} \circ \Psi^{-1})(x,y)= x^2 + y^2$ for all $(x,y) \in \Psi(U)$.
\end{enumerate}
\end{Proposition}

\begin{proof}
Since $p$ has tall local model $Y = T \times_H (\fh^{\circ} \times \C^{h+1})$, Remark \ref{rmk:brilliant} implies that $(1,\ldots,
1) \notin \rho_*(\fh)$, where $\rho \colon H \to (S^1)^{h+1}$ is the symplectic slice representation at $p$. Hence, by Lemma \ref{need} and the Morse Lemma, after possibly replacing $g$ by $-g$, we may assume that $M$ is a local model $Y = T \times_H (\fh^\circ \times \C^{h+1})$, that $\Phi = \Phi_Y$, that $g([t,\alpha,z]) = |z|^2$, and that $p = [1,0,0]$. Let $N$ be the degree of the defining monomial $P\colon \C^{h+1} \to \C$.
By Lemma \ref{trivial}, the map $\overline{P} \colon \YmodT \to \C$ is a
homeomorphism; moreover, it
restricts to a diffeomorphism from $(\YmodT)\smallsetminus \{ [1,0,0] \} $ to  $\C \smallsetminus \{0\}$.
By Lemma \ref{lemma:constant}, there exists a positive real number $C$ such that $g([t,\alpha,z]) = C |P(z)|^\frac{2}{N}$ for all $[t,\alpha,z] \in \Phi_Y^{-1}(0)$.
Let $S \colon \C \to \C$ be the homeomorphism that, in polar coordinates, 
takes $(r,\theta)$ to $(\sqrt{C} \, r^{\frac{1}{N}}, \theta)$; it restricts
to a diffeomorphism $\C \smallsetminus \{0\} \to \C \smallsetminus \{0\}$.
Define $\Psi \colon \YmodT \to \C \simeq \R^2$ by $\Psi([t,\alpha,z]) =
S(P(z))$. Since 
$$ |\Psi([t,\alpha,z])|^2 = C|P(z)|^{\frac{2}{N}} = \overline{g}([t,\alpha,z]), $$
\noindent
for all $[t,\alpha,z] \in \YmodT$, the map $\Psi$ satisfies properties (1) -- (3).
\end{proof}

\begin{Remark}
    We observe that, by Lemma \ref{trivial} (and \ref{lemma:criterion-exc}), if $N = 1$ the map $\Psi$ constructed in the above proof is a diffeomorphism.
\end{Remark}

\section{The semitoric case}\label{sec:semitoric}
The goal of this section is to prove that focus-focus points satisfy the hypotheses of Proposition \ref{connected} (see Lemma \ref{focusfocus} below). This allows us to provide a short and direct new proof of Vu Ngoc's theorem that the fibers of semitoric systems are connected \cite{VN} -- see Theorem \ref{thm:semitoric}. In both proofs, the key arguments follow from Morse theory; the most significant difference is that we construct a new smooth structure on the reduced spaces, while he uses the natural smooth structure on $M$. 

\begin{Definition}\label{semitoric}
  An integrable system $\left(M,\omega,\ft
    \times \R, f=\left(\Phi,g\right)\right)$ is {\bf
    semitoric} if 
  \begin{itemize}
  \item $\left(M,\omega,\Phi\right)$ is a 4-dimensional complexity one $T$-space, and 
  \item every singular point of $\left(M,\omega,\ft \times \R,f\right)$ is non-degenerate with no hyperbolic blocks.
  \end{itemize}
\end{Definition}

Note that, by definition, every singular point of a semitoric system either has purely elliptic type or is focus-focus, i.e., it has one focus-focus block. 
We begin with an important example of a semitoric system containing one focus-focus point.

\begin{Example}\label{example:ff_model}
Let $Y = \C^2$, let $\omega_Y$ be the standard symplectic form, and set $\Phi_Y(z_1,z_2) := \frac{1}{2}\left(|z_1|^2-|z_2|^2\right)$ and $g_Y(z_1,z_2):=\Im(z_1 z_2)= \Im(P(z_1,z_2))$, where $P \colon Y \to \C$ is the defining monomial for the $S^1$-action induced by $\Phi_Y$ (cf. Example \ref{ex2}).
\end{Example}

Locally, near a focus-focus point, every semitoric system is modeled on Example \ref{example:ff_model}.

\begin{Proposition}\label{lemma:focus-focus}
  Let $(M,\omega,\ft \times \R, f = (\Phi,g))$ be an integrable system
  such that  $(M,\omega,\Phi)$ is a 4-dimensional complexity one $T$-space.
  Assume that $p \in M$ is focus-focus. 
  After identifying $T$ with $S^1$, a neighborhood of $p$ is isomorphic to a neighborhood of the origin in the semitoric system $(Y,\omega_Y,\R \times \R, f_Y)$ described above.
\end{Proposition}

\begin{proof}
Identify $T$ with $S^1$. By the local normal form for focus-focus points, we may assume that the integrable system is  $(Y,\omega_Y,\R \times \R, \varphi \circ f_Y)$ and $p$ is the origin, where $(Y,\omega_Y,\R \times \R, f_Y = (\Phi_Y,g_Y))$ is as in Example \ref{example:ff_model} and $\varphi \colon \R \times \R \to \R \times \R$ is a diffeomorphism \cite{C,VW}. Since by assumption the first component of $\varphi \circ f_Y$ induces an effective $S^1$-action on $Y$, and since the fibers of $f_Y$ are not compact, up to a translation $\Phi_Y$ and the first component of $\varphi \circ f_Y$ are equal up to sign (cf. \cite[Proposition 3.9]{HSS}). Since the integrable systems $(Y,\omega_Y,\R \times \R, f_Y = (\Phi_Y,g_Y))$ and $(Y,\omega_Y,\R \times \R, (-\Phi_Y,g_Y))$ are isomorphic, the claim follows.
\end{proof}

Hence, focus-focus points satisfy the hypotheses of Proposition \ref{connected}.

\begin{Lemma}\label{focusfocus}
Let $(M,\omega,\ft \times \R, f = (\Phi,g))$ be an integrable system such that  $(M,\omega,\Phi)$ is a 4-dimensional complexity one $T$-space. Assume that $p \in \Phi\inv(0) \cap g\inv(0)$ is focus-focus. 
Then
there exist an open neighborhood $U$ of $[p] \in \MmodT $
and a map $\Psi \colon U \to \R^2$ taking $[p]$ to $(0,0)$ satisfying:
\begin{enumerate}
\item $\Psi \colon U \to \Psi(U)$ is a homeomorphism onto an open set;
\item $\Psi$ restricts to a diffeomorphism on $U \smallsetminus \{[p]\}$; and
\item $(\overline{g} \circ \Psi^{-1})(x,y) = y$ for all $(x,y) \in \Psi(U)$.
\end{enumerate}
\end{Lemma}

\begin{proof}
By Proposition \ref{lemma:focus-focus} we may assume that the integrable system is the system $(Y,\omega_Y,\R \times \R, f_Y:=(\Phi_Y,g_Y))$ described in Example \ref{example:ff_model}, and that $p$ is the origin. Set $\Psi[z_1,z_2]:= z_1z_2 = P(z_1,z_2)$. Then the claim follows immediately from Lemma \ref{trivial}.
\end{proof}

We are now almost ready to give our proof of Vu Ngoc's theorem on the connectedness of the fibers of semitoric systems; however,
we also need the following important result, which is proved in \cite{At,GS2,LMTW}.

\begin{Theorem}\label{thm:connected}
  Let $T$ act on a connected symplectic manifold $(M,\omega)$ with moment map $\Phi \colon M \to \ft^*$. If $\Phi$ is proper, then every fiber of $\Phi$ is connected and $\Phi$ is open as a map to $\Phi(M)$. 
\end{Theorem}

We also need the following proposition. Atiyah proved the compact case (see \cite{At}), and used it to prove the theorem above;
it was later extended to the proper case; see \cite[Proposition 4.5]{PRV}.  (For completeness, we provide a variation on the arguments in \cite{PRV} in Section \ref{sec:proofs}.)

\begin{Proposition}\label{prop:connected}
    Let $N$ be a connected  manifold and let $\overline{g} \colon N \to \R$ be a proper Morse-Bott function. If no critical point of $\overline{g}$ has index or coindex one, then the fiber $\overline{g}^{-1}(c)$ is connected for all $c \in \R$.
\end{Proposition}

We observe that the result below is a special case of Theorem \ref{thm:equivalence}, and the proof below illustrates our general strategy.

\begin{Theorem}[V\~u Ngoc]\label{thm:semitoric}
  Let $(M,\omega,\ft \times \R, f = (\Phi,g))$ be a semitoric system such that $\Phi$ is proper. 
   Then the fiber $f^{-1}(\beta,c) = \Phi^{-1}(\beta) \cap g^{-1}(c)$ is connected for all $(\beta,c) \in \ft^* \times \R$.
\end{Theorem}

\begin{proof}
 Fix $\beta \in \ft^*$.  Assume without loss of generality that $\beta = 0$. Since $\Phi$ is proper, the reduced space $\MmodT$ is compact and connected by Theorem~\ref{thm:connected}.

 If the reduced space $\MmodT$ is a single point 
 then obviously $\Phi^{-1}(0)$ consists of a single $T$-orbit. Hence, the fiber
 $f^{-1}(0,c)$ is connected for all $c \in \R$. 

Otherwise, since the orbit of a short 
point in $\Phi^{-1}(0)$ is an isolated point in $\Phi^{-1}(0)/T$, every point in $\Phi^{-1}(0)$ is tall. If $p \in \Phi^{-1}(0)$ is a critical point of $g$ modulo $\Phi$, then $p$ is also a singular point of $\left(M,\omega, \ft \times \R, f\right)$. So, by assumption, $p$ either has purely elliptic type or is focus-focus. By Propositions \ref{connected} and \ref{gcrit}, and Lemma \ref{focusfocus}, $\MmodT$ can be given the structure of a smooth closed, connected, oriented surface such that $\overline{g} : \MmodT \to \R$ is a Morse function with no critical point of index 1. By Proposition \ref{prop:connected}, the fiber $\overline{g}^{-1}(c)$ is connected for all $c \in \R$. Since $T$ is connected, it follows that $f^{-1}(0,c)$ is also connected.
 \end{proof}

\section{Non-degenerate singular points}\label{sec:non-degenerate}

Let $\left(M,\omega, \ft \times \R, f:=\left(\Phi,g\right)\right)$ be an integrable  system such that $\left(M,\omega,\Phi\right)$ is a complexity one $T$-space. Let $p \in \Phi^{-1}(0)$ be a tall non-degenerate singular point. By Lemma \ref{greg2} and Proposition \ref{gcrit}, if $p$ has purely elliptic type, then we can put a smooth structure on $\MmodT$ near $[p]$ so that $\overline{g}$ is locally Morse, and hence $p$ satisfies the hypotheses of Proposition \ref{connected}. The aim of this section is to extend this result to the remaining cases (see Proposition \ref{nondegenerate}). In particular, if $p$ either has a focus-focus block, or has a hyperbolic block and disconnected $T$-stabilizer, then $[p]$ is a regular point of $\overline{g}$, while if $p $ has a hyperbolic block and connected $T$-stabilizer, then $[p]$ is a singular point with index one. We start by recalling some facts about non-degenerate singular points.
\begin{Remark}\label{rmk:non-degenerate}
  Let $p$ be a point in an integrable system $(M,\omega,V,f)$ and let $W_0 \subseteq V$ be the Lie algebra of the stabilizer of $p$. Let $\mathcal{O}$ be the $V$-orbit through $p$. The linearized action at $p$ induces a Lie algebra homomorphism $W_0 \to \mathfrak{sp}((T_p\rorbit)^\omega/T_p \rorbit)$. Given $\xi \in W_0$, we can calculate the associated symplectic endomorphism as follows: Let $f_{\mathrm{lin}}$ be the homogeneous moment map for the linearized action. 
  We view $\langle f_{\mathrm{lin}},\xi \rangle$ as a homogeneous quadratic polynomial on $(T_p\rorbit)^\omega/T_p \rorbit$.
  We recall that there is a standard isomorphism between homogeneous quadratic polynomials and $\mathfrak{sp}((T_p\rorbit)^\omega/T_p \rorbit)$ that, upon choosing a standard basis for the symplectic vector space $(T_p\rorbit)^\omega/T_p \rorbit$, sends a symmetric matrix $A$ to $JA$, where $J$ is the standard complex structure on $(T_p\rorbit)^\omega/T_p \rorbit$. The above Lie algebra homomorphism sends $\xi$ to  $J\langle f_{\mathrm{lin}},\xi \rangle$.
  
  If, in addition, $p$ is non-degenerate, then for a generic element $\xi$ in the Lie algebra of $W$, the eigenvalues of the above endomorphism of $(T_p\rorbit)^\omega/T_p \rorbit$ can be characterized as follows: If $f_{\mathrm{lin}}$ can be written as the product of $k_e$ elliptic blocks, $k_h$ hyperbolic blocks and $k_f$ focus-focus blocks, then the above endomorphism has exactly $2k_e$ distinct non-zero imaginary eigenvalues, $2k_h$ distinct non-zero real eigenvalues, and $4k_f$ distinct non-zero complex eigenvalues that are neither real nor imaginary. In particular, the numbers $k_e, k_h$ and $k_f$ are a well-defined invariant of $p$.
\end{Remark}

 Before proceeding to the tall case, we start by analyzing short points.
\begin{Lemma}\label{short_elliptic}
Let $\left(M,\omega, \ft \times \R, f:=\left(\Phi,g\right)\right)$ be an integrable system such that $\left(M,\omega,\Phi\right)$ is a complexity one $T$-space. If $p \in M$ is short and non-degenerate, then $p$ has purely elliptic type.
\end{Lemma}

\begin{proof}
For simplicity, assume that $g(p) =0$. By the Marle-Guillemin-Sternberg local normal form theorem, we may assume that $M$ is a local model $Y = T \times_H (\fh^\circ \times \C^{h+1})$, that $\Phi = \Phi_Y$, and that $p =
[1,0,0]$. Let $j \colon \C^{h+1} \to T \times_H \fh^\circ \times \C^{h+1} $ denote the inclusion $z \mapsto [1,0,z]$. By Lemma \ref{greg2}, $p$ is a critical point of $g$ modulo $\Phi$. Hence, the $(T \times \R)$-orbit $\rorbit$ through $p$ coincides with the $T$-orbit $\mathfrak{O}$ through $p$ and the Taylor polynomial $T^1_0 (j^*g)$ of degree one at $0 \in \C^{h+1}$ vanishes. Thus we may identify $(T_p\rorbit)^\omega/T_p \rorbit = (T_p\mathfrak{O})^\omega/T_p \mathfrak{O}$ with  $\C^{h+1}$ so that the homogeneous moment map for the linearized action at $p$ is the Taylor polynomial of degree two $T^2_0 (j^*f)= (T^2_0(j^*\Phi),T^2_0 (j^*g)) = (\Phi_H, T^2_0 (j^*g))$, where $\Phi_H \colon \C^{h+1} \to \fh^*$ is the homogeneous moment map for the $H$-action. 

 By \cite[Lemma 5.4]{KT1}, $\Phi$ is proper; hence, so is $(\Phi_H, T_0^2 (j^*g))$. If $p$ has a hyperbolic or a focus-focus block, then the homogeneous moment map for the linearized $(H \times \R)$-action on $\C^{h+1}$ is not proper. Hence, $p$ has  purely elliptic type.
\end{proof}

We can now state the main result of this section.
\begin{Proposition}\label{nondegenerate}
Let $\left(M,\omega, \ft \times \R, f:=\left(\Phi,g\right)\right)$ be an integrable  system such that $\left(M,\omega,\Phi\right)$ is a complexity one $T$-space. Assume that  $p \in \Phi^{-1}(0) \cap g^{-1}(0)$ is non-degenerate and tall; let $N$ be the degree of  the defining monomial of the local model for $p$.
Then $p$  has at most one non-elliptic block; furthermore,
\begin{enumerate}
\item If $p$ has a hyperbolic block and connected $T$-stabilizer, then $N = 1$. Moreover, there exist an open neighborhood $U$ of $[p] \in \MmodT$ and a map $\Psi \colon U \to \R^2$ taking $[p]$ to $(0,0)$ satisfying:
\begin{enumerate}[label=(\alph*),ref=(\alph*)]
    \item \label{item:-500} $\Psi \colon U \to \Psi(U)$ is a diffeomorphism onto an open set; and
    \item \label{item:-600} $(\overline{g} \circ \Psi^{-1})(x,y) = x^2 - y^2$ for all $(x,y) \in \Psi(U)$.
   \end{enumerate}
    \item If  $p$ either has a focus-focus block, or  has a hyperbolic block and disconnected $T$-stabilizer, then $N = 2$. Moreover,  there exist an open neighborhood $U$ of $[p] \in \MmodT$ 
  and a map $\Psi \colon U \to \R^2$ taking $[p]$ to $(0,0)$ satisfying:
  \begin{enumerate}
  \item \label{item:800} $\Psi \colon U \to \Psi(U)$ is a homeomorphism onto an open set;
  \item \label{item:1200} $\Psi$ restricts to a diffeomorphism on $U \smallsetminus \{[p]\}$; and
      \item \label{item:1300} $(\overline{g} \circ \Psi^{-1})(x,y)= y$ for all $(x,y) \in \Psi(U)$.
  \end{enumerate}
  \end{enumerate}
\end{Proposition}

\begin{proof}
By the Marle-Guillemin-Sternberg local normal form theorem, we may assume that $M$ is a tall local model $Y = T \times_H \fh^\circ \times \C^{h+1}$, that $\Phi = \Phi_Y$, and that $p =
[1,0,0]$. Let $j \colon \C^{h+1} \to T \times_H \fh^\circ \times \C^{h+1} $ denote the inclusion $z \mapsto [1,0,z]$. 

 By Lemma~\ref{greg2}, every regular point of $g$ modulo $\Phi$ has purely elliptic type.
Therefore, we may assume that $p$ is a critical point of $g$ modulo $\Phi$.
As in the proof of Lemma \ref{short_elliptic}, this implies that the Taylor polynomial  of degree one $T^1_0(j^*g)$ is identically zero, and we may identify $(T_p\rorbit)^\omega/T_p \rorbit$ with  $\C^{h+1}$ so that the homogeneous moment map for the linearized action at $p$ is the Taylor polynomial of degree two $(\Phi_H, T^2_0 (j^*g))$. 

Let $P \colon \C^{h+1} \to \C$ be the defining monomial of $Y$. Depending on the degree $N$ of $P$, we may assume that we are in one of the following four cases (cf.\ Examples \ref{ex1} and \ref{ex2}). \\

{\bf Case 1: $\mathbf{N=1}$ and $\mathbf{P(z) = z_0}$.} In this case, $\C^{h+1} \simeq \C \times \C^{h}$ and $H \simeq (S^1)^h$, where $(S^1)^h$ acts on $\C^{h}$ in the standard fashion and trivially on $\C$. Thus the span of   $\{\Phi_H^{\xi} \mid \xi \in \fh\}$   equals the span of 
 $\{|z_k|^2 \colon k = 1,\ldots, h\}.$
Moreover,  since the Taylor polynomial $T^2_0 (j^*g)$ is $H$-invariant, by Remark \ref{rmk:other_basis} it   lies in the span of $\{\Re(z_0^2),\Im(z_0^2), |z_k|^2 \colon k=0,\ldots,h \}$.
Therefore there exist constants $a_x,a_y, a_\rho \in \R$ and  $\xi \in \fh$  such that 
\begin{equation}
\label{eq:-20}
 T^2_0(j^*g) (z)=  a_x \Re(z_0^2) + a_y \Im (z_0^2) + a_{\rho} |z_0|^2 + \Phi_H^{\xi} (z) \  \forall  z \in \C^{h+1}.
\end{equation}
Moreover, a generic element in the span of $\{T^2_0 (j^*g), \Phi_H^{\xi} \mid \xi \in \fh \}$ is a multiple of
$$     a_x \Re(z^2_0) + a_y \Im (z^2_0) + a_{\rho} |z_0|^2 + \sum\limits_{k=1}^{h} c_k |z_k|^2, 
$$
\noindent
where $c_k \in \R \smallsetminus \{0\}$ for $k=1,\ldots, h$. The eigenvalues of the corresponding symplectic endomorphism are given by 
\begin{equation*}
      \pm 2\sqrt{a_x^2+a_y^2 - a_{\rho}^2},\,\, \pm 2c_1 \sqrt{-1},\,\, \ldots,\, \,\pm 2c_{h}\sqrt{-1}.
\end{equation*}
 Hence, since $p$ is non-degenerate,
 Remark \ref{rmk:non-degenerate} implies that 
 $a_x^2+a_y^2 - a_{\rho}^2 \neq 0$. Moreover, if $a_x^2+a_y^2 - a_{\rho}^2 > 0$, then $p$ has one hyperbolic block and the remaining blocks are elliptic; otherwise, $p$ has purely elliptic type.

By Lemma \ref{trivial}, $\overline{P}$ is a diffeomorphism; moreover,  the inclusion $i \colon \C \to Y$ sending 
$z_0$ to $[1,0,(z_0,0,\ldots, 0)]$
induces the inverse diffeomorphism $\overline{i} \colon \C \to \YmodT$. 
By definition of the inclusion $j$ and by  \eqref{eq:-20}, 
$$ T^2_0 (\overline{i}^*\overline{g}) (z_0) = a_x \Re(z_0^2) + a_y \Im (z_0^2) + a_{\rho} |z_0|^2 \quad  \forall  z_0 \in \C. $$
Hence, the zero set of the map $\YmodT \to \R$ taking $[t,\alpha,z] \in \YmodT$ to $T^2_0 (j^*g)(z)$ is homeomorphic to the set $\{(x,y) \in \R^2 \mid xy=0\}$ if $p$ has a hyperbolic block, and is a single point if $p$ has purely elliptic type. 

Finally, assume that $p$ has a hyperbolic block. Then $0$ is a non-degenerate critical point of $\overline{i}^*\overline{g}$ of index 1. Hence, by the Morse lemma, there exists an open neighborhood $V$ of $(0,0) \in \R^2$ and a map $\psi \colon V \to \C$ taking $(0,0)$ to $0$ that is a diffeomorphism onto an open set 
 such that $(\overline{i} \circ \psi)^* \overline{g} (x,y) = x^2 - y^2$ for all $(x,y) \in V$. Setting $U:=(\overline{i} \circ \psi)(V)$ and $\Psi \colon U \to \R^2$ to be $(\overline{i}\circ \psi)^{-1},$ the result follows. \\

 {\bf Case 2: $\mathbf{N=2}$ and $\mathbf{P(z) = z_0^2}$.} In this case, $\C^{h+1} \simeq \C \times \C^{h}$ and $H \simeq \mathbb{Z}_2 \times (S^1)^h$, where $\mathbb{Z}_2$ acts on $\C$ by multiplication with $-1$ and trivially on $\C^{h}$, and $(S^1)^h$ acts on $\C^{h}$ in the standard fashion and trivially on $\C$. Thus the span of $\{\Phi_H^{\xi} \mid \xi \in \fh\}$ equals the span of $\{|z_k|^2 \colon k = 1,\ldots, h\}$. 
Moreover, since the Taylor polynomial $T^2_0 (j^*g)$ is $H$-invariant, by Remark \ref{rmk:other_basis} it lies in the span of $\{\Re (z_0^2), \Im (z_0^2),  |z_k|^2 \colon k=0,\ldots,h \}$. Hence, there exist constants $a_x,a_y, a_\rho \in \R$ and $\xi \in \fh$ such that
    \begin{equation}
    \label{eq:hyper}
    T^2_0(j^*g)(z) = a_x \Re(z_0^2) + a_y \Im (z_0^2) + a_{\rho} |z_0|^2 + \Phi^{\xi}(z) \ \forall z \in \C^{h+1}.
\end{equation}
As in the proof of Case 1, since $p$ is non-degenerate, Remark \ref{rmk:non-degenerate} implies that $a_x^2+a_y^2 - a_{\rho}^2 \neq 0$. Moreover, if $a_x^2+a_y^2 - a_{\rho}^2 > 0$, then $p$ has one hyperbolic block and the remaining blocks are all elliptic; otherwise, $p$ has purely elliptic type. 

By Lemma \ref{lemma:hom}, the inclusion $i \colon \C \hookrightarrow Y$ sending $z_0$ to $[1,0,(z_0,0,\ldots, 0)]$
induces a diffeomorphism $\overline{i}$ between  $\C/\!/\Z_2 = \C/\Z_2$ and $\YmodT$. By definition of the inclusion $j$ and by \eqref{eq:hyper}, 
$$ T^2_0 (i^*g)(z_0) = a_x \Re(z_0^2) + a_y \Im (z_0^2) + a_{\rho} |z_0|^2 \quad \forall z_0 \in \C. $$
 This implies that the zero set of the map $\YmodT \to \R$ taking $[t,\alpha,z] \in \YmodT$ to $T^2_0 (j^*g)(z)$ is homeomorphic to $\R$ exactly if $p$ has a hyperbolic block, and is a single point if $p$ has purely elliptic type. 

Finally, assume that $p$ has a hyperbolic block. Then $0$ is a non-degenerate critical point of $i^*g$ of index 1. Hence, by the equivariant Morse lemma, there exists an open neighborhood $V$ of $(0,0) \in \R^2 \simeq \C$ and a map $\psi \colon V \to \C$ taking $(0,0)$ to $0$ that is a $\Z_2$-equivariant diffeomorphism onto an open set 
 such that $(i \circ \psi)^* g (x,y) = xy$ for all $(x,y) \in V$. Let $\pi \colon \R^2/\Z_2 \to \R^2$ be the map that sends $[x,y]$ to $(xy,x^2-y^2)$; it is a homeomorphism onto $\R^2$ and a diffeomorphism on the complement on $(0,0)$. Setting $U:=(\overline{i} \circ \overline{\psi})(V/\Z_2)$ and $\Psi \colon U \to \R^2$ to be $ \pi \circ(\overline{i}\circ \overline{\psi})^{-1},$ the result follows. \\

{\bf Case 3: $\mathbf{N = 2}$ and $\mathbf{P(z) = z_0z_1}$.} In this case, $\C^{h+1} \simeq \C^2 \times \C^{h-1}$ and $H \simeq S^1 \times (S^1)^{h-1}$, where $S^1$ acts on $\C^2$ with weights $\pm 1$ and trivially on $\C^{h-1}$, and $(S^1)^{h-1}$ acts on $\C^{h-1}$ in the standard fashion and trivially on $\C^2$. Hence, the span of $\{\Phi_H^{\xi} \mid \xi \in \fh\}$ equals the span of $\{|z_k|^2 \colon k =2, \ldots, h\}$ and $\Phi_1(z) := \frac{1}{2}(|z_0|^2-|z_1|^2)$. Moreover, since the Taylor polynomial $T^2_0 (j^*g)$ is $H$-invariant, by Remark \ref{rmk:other_basis} it lies in the span of $\{\Re(z_0 z_1),\Im(z_0 z_1), |z_k|^2 \colon   k=0,\ldots,h \}$. Therefore, there exist constants $a_x,a_y, a_\rho \in \R$ and $\xi \in \fh$ such that 
\begin{equation}
\label{eq:ffgeneral2}
     T^2_0(j^*g)(z) = a_x \Re(z_0z_1) + a_y \Im (z_0z_1) + a_{\rho} (|z_0|^2 +|z_1|^2)+ \Phi^{\xi}(z) 
\end{equation}
\noindent
for all $z \in \C^{h+1}$. Moreover, a generic element in the span of $\{T^2_0 (j^*g), \Phi_H^{\xi} \mid \xi \in \fh\}$ is a multiple of
$$ a_x \Re(z_0z_1) + a_y \Im (z_0z_1) + a_{\rho} (|z_0|^2+|z_1|^2) + c_1 \Phi_1 + \sum\limits_{k=2}^{h} c_k |z_k|^2, $$
\noindent
where $c_j \in \R \smallsetminus \{0\}$ for $j=1,\ldots, h$. The eigenvalues of the corresponding symplectic endomorphism are given by 
\begin{equation*}
        \pm 2c_1 \sqrt{-1} \pm \sqrt{a_x^2 + a_y^2 - 4a_{\rho}^2}\, , \, \pm 2c_2 \sqrt{-1},\,\ldots,\, \pm 2c_{h} \sqrt{-1},
\end{equation*}
\noindent
where the signs in the first term are independent. Hence, since $p$ is non-degenerate, Remark \ref{rmk:non-degenerate} implies that $a_x^2 + a_y^2 - 4a_{\rho}^2 \neq 0$. If $a_x^2 + a_y^2 - 4a_{\rho}^2 > 0$,  then $p$ has one focus-focus block, and the remaining blocks are elliptic; otherwise $p$ has purely elliptic type. It follows from \eqref{eq:ffgeneral2}, and Lemmas \ref{trivial} and \ref{lemma:constant}, that the defining monomial identifies the map $\YmodT \to \R$ taking $[t,\alpha,z] \in \YmodT$ to $T^2_0 (j^*g)(z)$ with
the map $\C \to \R$ taking $w \in \C$ to $a_x \Re(w) + a_y \Im(w) + 2a_\rho |w|$. Hence, the zero set of the map $\YmodT \to \R$ taking $[t,\alpha,z] \in \YmodT$ to $T^2_0 (j^*g)(z)$ is homeomorphic to $\R$ exactly if $p$ has a focus-focus block, and is a single point if $p$ has purely elliptic type. 

Finally, assume that $p$ has a focus-focus block. 
Let $i \colon \C^2 \hookrightarrow Y$ be the inclusion sending $(z_0,z_1)$ to $[1,0,(z_0,z_1,0,\ldots, 0)]$.
We observe that $i^*\omega_Y$ is the standard symplectic form on $\C^2$; see Remark \ref{rmk:symp-form}. Moreover, since $i(\C^2)$ is invariant under the $S^1$-action generated by $\Phi_1$, the functions $i^*\Phi_1$ and $i^*g$ commute with respect to the induced Poisson bracket.
By definition of the inclusion $j$ and by \eqref{eq:ffgeneral2}, there exists  $c \in \R$ such that
\begin{equation*}
    T^2_0 (i^*g)(z) = a_x \Re(z_0z_1) + a_y \Im (z_0z_1) + a_{\rho} (|z_0|^2 + |z_1|^2) + c \,i^*\Phi_1(z)
\end{equation*}
for all $z = (z_0,z_1) \in \C^2$. 
 Since $a_x, a_y,$ and $ a_{\rho}$ are not all zero, the quadratic forms $i^*\Phi_1$ and $T^2_0 (i^* g)$ are linearly independent.   Hence the map $(i^*\Phi_1,T^2_0 (i^*g)) \colon \C^2 \to \R^2$ is regular on a dense subset.  Since the Taylor polynomial of $i^*g - T^2_0 (i^*g)$ of degree two at zero vanishes, the 2-form $d (i^*\Phi_1) \wedge d (i^*g -T^2_0 (i^*g))$ goes to 0 more quickly than $d (i^*\Phi_1) \wedge d (T^2_0 (i^*g))$ as $z = (z_0,z_1)$ approaches $0$. Hence,  the map 
 $(i^*\Phi_1,i^*g)$ is regular on a dense subset of an open neighborhood $W$ of $0$ and so $(W,i^*\omega_Y, \R \times \R, (i^*\Phi_1,i^*g))$ is an integrable system. Moreover, since $a_x^2 + a_y^2 - 4a_{\rho}^2 > 0$, the point
 $0 \in \C^2$ is a non-degenerate singular point with one focus-focus block. 
So by Lemma \ref{focusfocus}, there exist an open neighborhood $V$ of $[0] \in \C^2/\!/S^1$ and a map $\psi \colon V \to \R^2$ taking $[0]$ to $(0,0)$ satisfying \eqref{item:800} -- \eqref{item:1300}.
Moreover, by Lemma \ref{lemma:hom},  the inclusion $i$ induces a diffeomorphism $\overline{i} \colon \C^2/\!/S^1 \to \MmodT$.
Setting $U:= \overline{i}(V)$ and $\Psi \colon U \to \R^2$ to be $\psi \circ \overline{i}^{-1}$, the result follows. \\

{\bf Case 4: $\mathbf{N > 2}$.} In this case, since the Taylor polynomial $T^2_0 (j^*g)$ is $H$-invariant, by Remark \ref{rmk:other_basis} it lies in the span of $\{ |z_k|^2 \colon k = 0, \dots, h \}$, as does $\Phi_H^\xi$ for any $\xi \in \fh$. Hence, any element in the span of $\{T^2_0 (j^*g), \Phi_H^{\xi} \mid \xi \in \fh\}$ can be written as 
\[\sum\limits_{k=0}^{h} c_k |z_k|^2. \]
The eigenvalues of the corresponding symplectic endomorphism are given by 
$$ \pm 2c_0 \sqrt{-1},\, \ldots,\, \pm 2c_{h}\sqrt{-1}. $$
Since $p$ is non-degenerate, the span of $\{T^2_0 (j^*g), \Phi_H^{\xi} \mid \xi \in \fh\}$ equals the span of $\{ |z_k|^2 \colon k = 0, \dots, h \}$. Hence, the zero set of the map $\YmodT \to \R$ taking $[t,\alpha,z]$ to $T^2_0(j^*g)$ is the point $[1,0,0]$. Moreover, since all the eigenvalues of the symplectic endomorphism corresponding to any element of the above span are imaginary, $p$ has purely elliptic type by Remark \ref{rmk:non-degenerate}.
\end{proof}

A careful reading of the proof of Proposition~\ref{nondegenerate} shows that it also
demonstrates the following result, which we will need in a future paper.

\begin{Lemma}\label{lemma:eph-ell}
    Let $(Y,\omega_Y,\ft \times \R, f:=(\Phi_Y,g))$ be an integrable system, where  $Y = T \times_H (\fh^{\circ} \times \C^{h+1})$ is a tall local model.
    Assume that $p = [1,0,0]$ is a non-degenerate critical point of $g$ modulo $\Phi$ with $g(p) = 0$. If  $j \colon \C^{h+1} \to Y$ is the inclusion $z \mapsto [1,0,z]$, then
\begin{itemize}
    \item the map $\YmodT \to \R$ taking $[t,\alpha,z] \in \YmodT$ to $T^1_0 (j^*g)(z)$ is identically zero, and
    \item the zero set of the map $\YmodT \to \R$ taking $[t,\alpha,z] \in \YmodT$ to $T^2_0 (j^*g)(z)$ is
    \begin{itemize}
        \item 
     a single point if $p$ has purely elliptic type,
     \item  homeomorphic to $\{(x,y) \in \R^2 \mid xy=0\}$ if $p$ has a hyperbolic block and connected $T$-stabilizer,
    and  \item homeomorphic to $\R$ otherwise.
    \end{itemize}
\end{itemize}
\end{Lemma}

\section{Proof of the main result}\label{sec:proofs}
In this section, we prove of the main results, Theorems \ref{thm:equivalence} and \ref{thm:converse}. 
To begin, we explain how certain questions about the topology of the integrable systems we consider can be reduced to understanding Morse theory for oriented surfaces.

\begin{Proposition}\label{prop:Morse}
Let $(M, \omega, \ft \times \R, f = (\Phi,g))$ be an integrable system such that $(M,\omega,\Phi)$ is a complexity one $T$-space.  Assume that each tall 
singular point in $(M,\omega,\ft \times \R, f)$ is non-degenerate. Fix $\beta \in \ft^*$.
Each component of the reduced space $\Phi^{-1}(\beta)/T$ containing more than one point
can be given the structure of a smooth  oriented 2-dimensional manifold such that $\overline{g} \colon \Phi^{-1}(\beta)/T \to \R$ is a Morse function. Moreover, given $p \in \Phi^{-1}(\beta)$:
\begin{enumerate}
\item $[p]$ is a critical point of $\overline{g}$ of index $1$ exactly if $p$ has a hyperbolic block and connected $T$-stabilizer,
\item $[p]$ is a critical point of $\overline{g}$ of index $0$ or $2$ exactly if $p$ is a critical point of $g$ modulo $\Phi$ with purely elliptic type, and
\item $[p]$ is a regular point of $\overline{g}$ exactly if $p$ has a focus-focus block, has a hyperbolic block with disconnected $T$-stabilizer, or is a regular point of $g$ modulo $\Phi$.
\end{enumerate}
\end{Proposition}
\begin{proof}
We may assume that $\beta = 0$. For simplicity, assume that $\Phi^{-1}(0)/T$ is connected and contains more than one point.
Then, by definition, every point in $\Phi^{-1}(0)$ is tall. 
If $p \in \Phi^{-1}(0)$ is a critical point of $g$ modulo $\Phi$,
then $p$ is also a singular point of $\left(M,\omega, \ft \times \R, f\right)$. So, by assumption, $p$ is non-degenerate. 
Therefore, by Proposition \ref{connected}, the claim follows from 
Propositions \ref{gcrit} and \ref{nondegenerate}.
\end{proof}

Therefore, we review a few lemmas from Morse theory.

\begin{Lemma}\label{uniquemin} Let $N$ be a connected manifold and let $\overline{g} \colon N \to \R$ be a proper Morse function. If  the fiber $\overline{g}^{-1}(c)$ is connected for each $c \in \R$, then $\overline{g}$ has at most one point of index $0$ (or coindex $0$).
\end{Lemma}

\begin{proof}
Let $p$ be a critical point of index $0$. We may assume that $\overline{g}(p)=0$.
By the Morse lemma, there exists a neighborhood $U$ of $p$
such that $\overline{g}^{-1}(0) \cap U = \{p\}$ and $\overline{g}(U) \subseteq [0,\infty)$.
Since the fibers are connected, the former implies
that $\overline{g}^{-1}(0) = \{p\}$.  Hence, the latter
implies that $\overline{g}^{-1}[0,\infty)$ is a nonempty  subset of $N$ that is both open and closed.  Since $N$ is connected, this implies that $\overline{g}$ achieves its minimum at $p$. Hence, $p$ is the unique point of index 0.
\end{proof}
\begin{Lemma}\label{Morse2}
    Let $N$ be a closed, connected 2-dimensional manifold and let $\overline{g} \colon N \to \R$ be a Morse function. The following are equivalent:
    \begin{enumerate}[label=(\arabic*),ref=(\arabic*)]
        \item no critical point of $\overline{g}$ has index one; and
        \item the fiber $\overline{g}^{-1}(c)$ is connected for all $c \in \R$ and $N$ is simply connected. 
    \end{enumerate}
\end{Lemma}

\begin{proof}
    Assume that no critical point of $\overline{g}$ has index one. Since $N$ is 2-dimensional, by Proposition \ref{prop:connected} the fiber $\overline{g}^{-1}(c)$ is connected for all $c \in \R$. Hence, since $N$ is closed, $\overline{g}$ has exactly two critical points by Lemma \ref{uniquemin}. By \cite[Theorem 4.1]{Mil}, $N$ is homeomorphic to a sphere and, therefore, simply connected.

    Conversely, assume that the fiber $\overline{g}^{-1}(c)$ is connected for all $c \in \R$ and that $N$ is simply connected. Since $N$ is closed, $\overline{g}$ has exactly one point of index zero and one point of index two by Lemma \ref{uniquemin}. Hence, by \cite[Theorem 5.2]{Mil}, the Euler characteristic of $N$ is equal to $2- C_1$, where $C_1$ is the number of critical points of $\overline{g}$ of index one. Since $N$ is closed and simply connected, it follows that $C_1 = 0$.
\end{proof}

We can now prove our main result.

\begin{proof}[Proof of Theorem \ref{thm:equivalence}]
 Let $\left(M,\omega, \ft \times \R, f:=(\Phi,g)\right)$ be an integrable system such that $(M,\omega,\Phi)$ is a complexity one $T$-space with a proper moment map. Assume that each tall 
singular point in $\left(M,\omega, \ft \times \R, f\right)$ is non-degenerate and that no such point has a hyperbolic block and connected $T$-stabilizer. Fix $\beta \in \ft^*$ 
such that $\Phi^{-1}(\beta)/T$ contains more than one point; we may assume that $\beta = 0$.
Since $\Phi$ is proper, the reduced space $\MmodT$ is compact and connected by Theorem~\ref{thm:connected}.
 Hence, every point in $\MmodT$ is tall, and so by Proposition~\ref{prop:Morse}, $\MmodT$ is a smooth closed,  connected oriented surface  so that 
$\overline{g} \colon \MmodT \to \R$
is a Morse function; moreover, $\overline{g}$ has no critical points of index $1$.  Therefore,
 by Lemma~\ref{Morse2}  the fiber $\overline{g}^{-1}(c)$ is connected
for all $c \in \R$ and  $\MmodT$ is simply connected.
Finally, since $T$ is connected, this implies that the fiber $f^{-1}(0,c)$ is connected for all $c \in \R$.
\end{proof}

To prove Theorem~\ref{thm:converse}, we also need the following result.

\begin{Lemma} \label{even} Let $N$ be a closed, connected oriented $2$-dimensional manifold
and let $\overline{g} \colon N \to \R$ be a Morse function. If the fiber $\overline{g}^{-1}(c)$ is connected for each $c \in \R$, then each fiber contains an even number of critical points of index $1$.
\end{Lemma}

\begin{proof}
By Lemma~\ref{uniquemin}, $N$ has a unique critical point of index $0$ (at the minimum), and a unique critical point of index $2$ (at the maximum).
Assume  that the fiber $\overline{g}^{-1}(c)$ contains $k \geq 1$ critical points of index $1$. 
By the Morse Lemma $c$ is in the interior of $\overline{g}(N)$,
and so there are no critical points of index $0$ or index $2$ in $\overline{g}^{-1}(c)$.  Fix $\epsilon > 0$ so that $[c - \epsilon, c + \epsilon]$ contains no other critical values.
Since  there is only one point of index $0$ we know that the preimage $\overline{g}^{-1}((-\infty, c - \epsilon])$ is connected. Hence, the first Betti number of  $\overline{g}^{-1}((-\infty, c + \epsilon])$
is the first Betti number of $\overline{g}^{-1}((-\infty, c - \epsilon])$ plus $k$.  Since every compact oriented $2$-dimensional manifold with boundary with an odd first Betti number has at least two boundary components, both of these Betti number must be even. Hence  $k$ is even.
\end{proof}

\begin{proof}[Proof of Theorem \ref{thm:converse}]
Let $\left(M,\omega, \ft \times \R, f:=(\Phi,g)\right)$ be an integrable system such that $(M,\omega,\Phi)$ is a complexity one $T$-space with a proper moment map. Assume that each tall 
singular point in $\left(M,\omega, \ft \times \R, f\right)$ is non-degenerate. 
Fix $\beta \in \ft^*$ such that $\Phi^{-1}(\beta)/T$ contains more than one point; we may assume that $\beta = 0$. Since $\Phi$ is proper, the reduced space $\MmodT$ is compact and connected by Theorem~\ref{thm:connected}.
Hence, every point in $\MmodT$ is tall, and so by Proposition~\ref{prop:Morse}, $\MmodT$ can be given the structure of a  smooth closed,  connected oriented surface  so that 
$\overline{g} \colon \MmodT \to \R$
is a Morse function with connected fibers; moreover, given $p \in \Phi^{-1}(0)$, the orbit $[p]$ is a critical point of $\overline{g}$ of index $1$  exactly if $p$ has a hyperbolic block and connected $T$-stabilizer.  Finally, since
$T$ is connected, the fiber $\overline{g}^{-1}(c)$ is connected exactly if 
the fiber $f^{-1}(\beta,c)$ is connected for any $c \in \R$.

If $\Phi^{-1}(\beta_0)/T$ is simply connected for some $\beta_0 \in \Phi(M)$, then $\Phi^{-1}(\beta')/T$ is simply connected for all $\beta' \in \Phi(M)$ by \cite[Lemma 5.7 and Corollary 9.7]{KT1}. In particular, $\MmodT$ is simply connected. Therefore,  the fibers $f^{-1}(\beta,c)$ are connected for all $c \in \R$ exactly if  $\MmodT$ has  no  tall singular  points with a hyperbolic block and connected $T$-stabilizer by Lemma~\ref{Morse2}.  

Similarly, if each fiber of $f$ contains at most one $T$-orbit of non-degenerate singular points with a hyperbolic block, then each fiber  of $\overline{g}$ contains at most one one critical point of index $1$.
Therefore, the fibers $f^{-1}(\beta,c)$ are connected for all $c \in \R$ exactly if $\MmodT$ has no tall singular points with a hyperbolic block and connected $T$-stabilizer  by Lemmas \ref{Morse2} and \ref{even}.
\end{proof}

The following result partially generalizes Theorem \ref{thm:equivalence}. 

\begin{Proposition}\label{prop:extra-result}
    Let $\left(M,\omega, \ft \times \R, f:=(\Phi,g)\right)$ be an integrable system such that $(M,\omega,\Phi)$ is a complexity one $T$-space with connected fibers. Assume that $f$ is proper, that each tall 
singular point in $\left(M,\omega, \ft \times \R, f\right)$ is non-degenerate, and that no such point has both a hyperbolic block and connected $T$-stabilizer. Then
 the fiber $f^{-1}(\beta,c)$ is connected for each $(\beta, c) \in \ft^* \times \R$.
\end{Proposition}

\begin{proof}
Fix $\beta \in \ft^*$ 
such that $\Phi^{-1}(\beta)/T$ contains more than one point; we may assume that $\beta = 0$.
By assumption, the reduced space $\MmodT$ is connected and $\overline{g} \colon \MmodT \to \R$ is proper.
 Hence, every point in $\MmodT$ is tall, and so by Proposition~\ref{prop:Morse}, $\MmodT$ is a smooth,  connected oriented surface  so that 
$\overline{g} \colon \MmodT \to \R$
is a Morse function; moreover, $\overline{g}$ has no critical points of index $1$. Therefore,
 by Proposition \ref{prop:connected}  the fiber $\overline{g}^{-1}(c)$ is connected
for all $c \in \R$.
Finally, since $T$ is connected, this implies that the fiber $f^{-1}(0,c)$ is connected for all $c \in \R$.
\end{proof}

Since Proposition~\ref{prop:connected} is of central importance, we conclude this section with a slight variation of the  proof given in \cite{PRV}. We begin with a preliminary lemma (cf. \cite[Lemma 4.3]{PRV}).

\begin{Lemma}\label{lem:connected} Let $N$ be a connected manifold and let $h \colon N \to \R$ be a proper function such that $h(N) \subseteq [0,\infty)$. Assume that the restriction of $h$ to $h^{-1}(0,\infty)$ is a Morse-Bott function with no points of index $1$.
 Then $h^{-1}(0)$ is connected.
\end{Lemma}

\begin{proof} Consider points $x$ and $y$ in $h^{-1}(0)$. Since $N$ is path-connected, there is a path $\gamma$ from $x$ to $y$. Hence, $x$ and $y$ are in the same component of $h^{-1}([0, b))$ for some $b \in \R$. If $C \subseteq h^{-1}((0,\infty))$ is a critical submanifold, then since the index of $C$ is not $1$, the restriction $H^0(h^{-1}([0,f(C) + \epsilon))) \to H^0(h^{-1}([0,f(C) - \epsilon)))$ is a surjection for sufficiently small $\epsilon > 0$.  In particular, if $x$ and $y$ are in the same component of $h^{-1}([0,h(C) + \epsilon))$, then they are also in the same component of $h^{-1}([0, h(C) - \epsilon))$.  Moreover, since $h$ is proper and its restriction to $h^{-1}(0, \infty)$ is Morse-Bott, the set of critical values in $(\delta,b)$ is finite for all regular $\delta \in (0,b)$. By  induction on these critical values,  $x$ and $y$ lie
in the same connected component $X_\delta$ of $h^{-1}([0,\delta])$ for all $\delta \in (0,b)$.  Since $N$ is Hausdorff,  and $X_\delta \subseteq X_{\delta'}$ are compact and connected for all $\delta' < \delta$,
the intersection $\cap_\delta X_\delta $ is connected.  
\end{proof}

\begin{proof}[Proof of Proposition \ref{prop:connected}]
Let $N$ be a connected manifold and let $\overline{g} \colon N \to \R$ be a proper Morse-Bott function. Fix $c \in \R$ and consider the proper function $h \colon N \to \R$ given by $h = (\overline{g} - c)^2$.  On the set $\overline{g}^{-1}(c, \infty)$ the function $h$ is Morse-Bott; the critical components of $h$ and $\overline{g}$ coincide and have the same index.  On the set $\overline{g}^{-1}(-\infty,c)$ the function $h$ is Morse-Bott; the critical components of $h$ and $\overline{g}$ coincide but the index and coindex of each critical manifold are exchanged. Therefore, if no critical point of $\overline{g}$ has index or coindex one, then $\overline{g}^{-1}(c) = h^{-1}(0)$ is connected by Lemma~\ref{lem:connected}.
\end{proof}


\begin{thebibliography}{1}

\bibitem[At]{At} M. F. Atiyah, {\bf Convexity and Commuting Hamiltonians}, Bull. London Math. Soc. 14 (1982), no. 1, 1 -- 15.
\bibitem[BF]{BF} A. T. Bolsinov and A. V. Fomenko, {\bf Integrable Hamiltonian systems. Geometry, Topology, Classification}, Chapman \& Hall, Boca Raton, 2004.

\bibitem[B]{B} G. E. Bredon, {\bf Introduction to compact transformation groups}, Academic Press, New York, 1972.

 \bibitem[C]{C} M. Chaperon, {\bf Normalisation of the smooth focus-focus: a simple proof}, Acta Math. Vietnam. 38 (2013), no. 1, 3 -- 9.
 
\bibitem[CVN]{CVN} Y. Colin de Verdière and S. V\~u Ngoc, {\bf Singular Bohr–Sommerfeld rules for
2D integrable systems}, Ann. Scient. Éc. Norm. Sup., 4e série, t. 36 (2003), 1 -- 55.

\bibitem[E]{E} L. H. Eliasson,  {\bf Normal forms for Hamiltonian
  systems with Poisson commuting integrals -- elliptic case}, Comment. Math. Helv. 65 (1990), 4 -- 35.


\bibitem[GS]{GS} V. Guillemin and S. Sternberg, {\bf A normal form for the moment map}, Differential geometric methods in mathematical physics
  ({J}erusalem, 1982), Math. Phys. Stud., 6, Reidel, Dordrecht,
  1984, 161 -- 175.

\bibitem[GS2]{GS2} V. Guillemin and S. Sternberg, {\bf Convexity properties of the moment mapping},  Invent. Math. 67 (1982), no. 3, 491 -- 513.
  

  
\bibitem[HSS]{HSS} S. Hohloch, S. Sabatini and D. Sepe,  {\bf From compact semi-toric systems to Hamiltonian $S^1$-spaces}, Discrete Contin. Dyn. Syst. 35 (2015), no. 1, 247 -- 281.

\bibitem[KT1]{KT1} Y. Karshon and S. Tolman, {\bf Centered complexity one {H}amiltonian torus actions}, Trans. Amer. Math. Soc. 353 (2001), no. 12, 4831 -- 4861.

\bibitem[KT2]{KT2} Y. Karshon and S. Tolman, {\bf Complete invariants for {H}amiltonian torus actions with two dimensional quotients}, J. Symplectic Geom. 2 (2003), no. 1, 25 -- 82.

\bibitem[KT3]{KT3} Y. Karshon and S. Tolman, {\bf Classification of {H}amiltonian torus actions with
    two-dimensional quotients}, Geom. Topol. 18 (2014), no. 2, 669 -- 716.


\bibitem[LMTW]{LMTW} E. Lerman, E. Meinrenken, S. Tolman and C. Woodward, {\bf Non-abelian convexity by symplectic cuts}, Topology 37 (1998), no. 2, 245 -- 259.

\bibitem[Li1]{Li} H. Li, {\bf The fundamental group of symplectic manifolds with Hamiltonian Lie group
actions}, J. Symplectic Geom. 4 (2006), no. 3, 345 -- 372.

\bibitem[Li2]{Li2} H. Li, {\bf The fundamental group of $G$-manifolds}, Commun. Contemp. Math. 15 (2013), no. 3, Article ID 1250056, 20 p.

\bibitem[M]{M} C.-M. Marle, {\bf Mod\`ele d'action hamiltonienne d'un groupe de {L}ie sur une
    vari\'{e}t\'{e} symplectique}, Rend. Sem. Mat. Univ. Politec. Torino 43 (1985), no. 2, 227 -- 251.

\bibitem[Mil]{Mil} J. Milnor, {\bf Morse Theory}, Princeton University Press, Princeton, 1961.



\bibitem[PRV]{PRV} \'A. Pelayo, T. Ratiu and S. V\~u Ngoc, {\bf Fiber connectivity and bifurcation diagrams of almost toric integrable systems}, J. Symplectic Geom. 13 (2015), no. 2, 343 -- 386. 




\bibitem[Sch]{Sch} G. Schwarz, {\bf Smooth functions invariant under the action of a compact Lie group}, Topology 14 (1975), 63 -- 68.


\bibitem[ST]{ST} D. Sepe and S. Tolman, {\bf Connectedness of fibers beyond semitoric systems II: ephemeral singular points}, preprint, arXiv:2510.16976.

\bibitem[SV]{SV} D. Sepe and S. V\~u Ngoc, {\bf Integrable systems, symmetries and quantization}, Lett. Math. Phys. 108 (2018), no. 3, 499 -- 571.


\bibitem[VN]{VN} S. V\~u Ngoc, {\bf Moment polytopes for symplectic manifolds with monodromy}, Adv. Math. 208 (2007), no. 2, 909 -- 934.

\bibitem[VW]{VW} S. V\~u Ngoc and C. Wacheux, {\bf Smooth normal forms for integrable Hamiltonian systems near a focus-focus singularity}, Acta Math. Vietnam. 38 (2013), no. 1, 107 -- 122.

\bibitem[Wa]{Wa} C. Wacheux, {\bf Syst\`emes int\'egrables semi-toriques et
polytopes moment}, Ph.D. Thesis, Universit\'e de Rennes 1, 2013.
\end{thebibliography}
\end{document}